\documentclass[11pt]{amsart}
\usepackage{geometry}                % See geometry.pdf to learn the layout options. There are lots.
\usepackage[parfill]{parskip}    % Activate to begin paragraphs with an empty line rather than an indent
\usepackage{amssymb, amsthm, amsmath, mathrsfs}
\usepackage{url}
\usepackage{xcolor}
\usepackage{comment}
\usepackage[all,cmtip]{xy}
\usepackage{graphicx}
\usepackage{xypic}
\usepackage{tikz}
\usepackage{bm}
\usepackage{float}
\usepackage{mathtools}
\usepackage{mathabx}
\usepackage{hyperref}

\makeatletter
\newtheorem*{rep@theorem}{\rep@title}
\newcommand{\newreptheorem}[2]{%
\newenvironment{rep#1}[1]{%
 \def\rep@title{#2 \ref{##1}}%
 \begin{rep@theorem}}%
 {\end{rep@theorem}}}
\makeatother

\newtheorem{theorem}{Theorem}[section]
\newtheorem{lemma}[theorem]{Lemma}
\newtheorem{corollary}[theorem]{Corollary}
\newtheorem{proposition}[theorem]{Proposition}
\newtheorem{conjecture}[theorem]{Conjecture}

\newreptheorem{theorem}{Theorem}

\theoremstyle{definition}
\newtheorem{definition}[theorem]{Definition}
\newtheorem{remark}[theorem]{Remark}
\newtheorem{example}[theorem]{Example}
\newtheorem{problem}[theorem]{Problem}

\def\rank{\textup{rank}}
\def\L{\mathcal{L}}
\def\H{\mathbb{H}}

\newcommand{\Q}{\mathbb{Q}}
\newcommand{\kk}{\Bbbk}

\newcommand{\Z}{\mathbb{Z}}
\newcommand{\F}{\mathbb{F}}

\newcommand{\A}{\mathfrak{A}}

\newcommand{\mfs}{\mathfrak{s}}
\newcommand{\mft}{\mathfrak{t}}
\newcommand{\mfw}{\mathfrak{w}}
\newcommand{\mfz}{\mathfrak{z}}

\newcommand{\spinc}{\text{Spin}^c}

\newcommand{\CFhat}{\widehat{CF}}
\newcommand{\CFminus}{CF^{-}}
\newcommand{\CFinfty}{CF^{\infty}}
\newcommand{\HFhat}{\widehat{HF}}
\newcommand{\HFminus}{HF^{-}}
\newcommand{\HFinfty}{HF^{\infty}}

\newcommand{\lk}{\ell k}
\newcommand{\PHS}{\mathrm{PHS}}
\newcommand{\dk}{d_{\kk}}
\newcommand{\B}{\mathrm{B}}
\newcommand{\Wh}{\mathrm{Wh}}

\title{Triple linking numbers and Heegaard Floer homology}
\author{Eugene Gorsky}
\address{Department of Mathematics, University of California, Davis, One Shields Avenue, Davis, CA 95616, USA}
\email{egorskiy@math.ucdavis.edu}

\author{Tye Lidman}
\address{Department of Mathematics, North Carolina State University, Raleigh, NC 27607, USA}
\email{tlid@math.ncsu.edu}

\author{Beibei Liu}
\address{Max Placnk Institute for Mathematics, Vivatsgasse 7, 53111 Bonn, Germany}
\email{bbliumath@gmail.com}

\author{Allison H. Moore}
\address{Department of Mathematics \& Applied Mathematics, Virginia Commonwealth University, 1015 Floyd Avenue, Box 842014, Richmond, VA 23284-2014, USA}
\email{moorea14@vcu.edu}

\begin{document}

\maketitle
\begin{abstract}
We establish some new relationships between Milnor invariants and Heegaard Floer homology.  This includes a formula for the Milnor triple linking number from the link Floer complex, detection results for the Whitehead link and Borromean rings, and a structural property of the $d$-invariants of surgeries on certain algebraically split links.      
\end{abstract}

\section{Introduction}

Milnor defined in his thesis a family of higher-order linking invariants capable of detecting triple linking, as in the Borromean rings \cite{Milnor-linkgroup}. The $\mu$-invariants are calculated using coefficients in the Magnus expansion of certain quotients of the fundamental group of the link complement. 
The Milnor invariants contain both link homotopy and concordance invariants of links \cite{Milnor, Casson, Stallings:homology}, and are central tools in the study of three-manifolds and four-manifolds.
Geometric interpretations of Milnor's invariants are numerous. 
Stallings conjectured that the $\mu$-invariants could be described in terms of Massey products for cohomology \cite{Stallings}, following which Turaev and Porter gave explicit interpretations \cite{Turaev, Porter}. 
Milnor's invariants can be computed using the intersection theory of certain ``derived" surfaces in the link exterior \cite{Cochran:memoirs} or expressed in terms of the Alexander and Conway polynomials \cite{Murasugi:Milnor, Traldi, Smythe}. The interpretation of the $\mu$-invariants most relevant to our purposes is the identification of the first non-vanishing coefficient of the Conway polynomial of an algebraically split link with the Sato-Levine invariant $\beta$ and square of the Milnor triple linking number $\mu_{123}$ for two- and three-component links, respectively, as determined by Cochran \cite{Cochran}. 

As many other invariants from knot theory can be seen in Floer homology, it is natural to ask about the Milnor invariants as well.  For example, it is asked in \cite[Problem 17.2.7]{Grid}:
\begin{problem}\label{problem:milnor-grid}
Do the Milnor invariants place algebraic restrictions on the
structure of link Floer homology?
\end{problem}
In previous work of the first, third, and fourth authors \cite{GLM}, it is shown that Heegaard Floer homology is able to see the Sato-Levine invariant $\beta$ of an algebraically split two-component link. In this paper, we address Problem~\ref{problem:milnor-grid} to study several appearances of the Milnor triple linking number $\mu_{123}$ \cite{Milnor} in the Heegaard Floer theory of links and three-manifolds.  We also apply this to give new Dehn surgery and link Floer complex detection results for the Whitehead link and the Borromean rings.  

The first result we have is that the link Floer complex of Ozsv\'ath-Szab\'o \cite{OS:linkpoly} contains the Milnor triple linking number.  
\begin{theorem}\label{thm:milnor}
Let $\L$ be a three-component link with pairwise linking number zero.  Then, there is an explicit formula for $|\mu_{123}(\L)|$ in terms of $CFL^-(\L)$.  
\end{theorem}
For the explicit statement, see Corollary \ref{cor:mucfl} below.   
The main strategy is to utilize the aforementioned relationship between the triple linking number and the Conway polynomial, and to express this in terms of the Euler characteristic of the link Floer complex and an associated link invariant. 

If we add some additional hypotheses on the links then this formula simplifies greatly.  The formula from Theorem~\ref{thm:milnor} becomes especially transparent if the link is a Brunnian L--space link.  First, recall that the $h$-function is an integer valued link invariant that is defined using absolute gradings in link Floer homology (see Section \ref{sec:hfunction} for a definition and properties). 
Second, recall that a link of three or more components is called Brunnian if all its proper sublinks are unlinks. 
In this article, we will call a two-component link Brunnian if it is linking number zero and the components are unknots, and will generally include unlinks in the class of Brunnian links. 
More generally, a link is called algebraically split if all pairwise linking numbers are zero. Finally, a link is called an L--space link if all sufficiently large surgeries of $S^3$ are L--spaces, that is, they have the simplest possible Heegaard Floer homology.  (In the first part of this section, we will work exclusively with $\mathbb{Z}_2$-coefficients.)

\begin{theorem}
Assume that $\L$ is a Brunnian L--space link with three components. Then
\label{triple squared}
$\mu_{123}(\L)^2 =\sum_\mathbf{s} h(\mathbf{s})$.
\end{theorem}
In particular, $\sum_\mathbf{s} h(\mathbf{s})$ is a link homotopy invariant of three-component Brunnian L--space links. See Corollary \ref{cor:mucfl} for a more general statement of homotopy invariance in the link Floer complex. 
For such links, the $h$-function is always non-negative (see Lemma \ref{h-function list} \eqref{h nonnegative}). As the unlink is the only L-space link with vanishing $h$-function (Lemma \ref{unlink1}), we have

\begin{corollary}
\label{unlink detection}
Assume that $\L$ is a Brunnian L-space link with three components. If $\mu_{123}=0$ then $\L$ is the three-component unlink. 
\end{corollary}

Next, we ask to what extent Floer homology is able to detect the simplest links admitting rationally framed Dehn surgery to the three-sphere or  the Poincar\'e homology sphere. We first observe that if there is a rational surgery $S^{3}_{1/m_{1}, \cdots, 1/m_{\ell}}(\L)$ on an $\ell$-component Brunnian link which is the three-sphere, then $\L$ is the unlink (see Proposition \ref{unlinksphere2}). 
We extend this by considering rational surgeries which are the Poincar\'e homology sphere and prove the following results: 

\begin{theorem}
\label{Brunnian detection from Poincare surgery combined}
Let $\L$ be an $\ell$-component Brunnian link, and suppose that $S^3_{1/{m_1}, \cdots, 1/{m_\ell}}(\L)$ is the Poincar\'e homology sphere.
    \begin{enumerate}
    \item \label{Brun 2} If $\ell=2$, then $\L$ is the Whitehead link or its mirror and $mn=1$. 
    \item \label{Brun 3} If $\ell=3$, then $\L$ is the Borromean rings, and all $|m_i|=1$ with the same sign. 
    \item \label{Brun 4} If $\ell\geq 4$, no such $\L$ exists.
    \end{enumerate}
\end{theorem} 
In a related vein, we prove that amongst Brunnian L-space links, the Whitehead link and Borromean rings are essentially detected by the Sato-Levine invariant and the triple linking number. 

\begin{theorem}
\label{Lspace link detection combined}
Let $\L$ be an $\ell$-component Brunnian L--space link. 
    \begin{enumerate}
    \item If $\ell=2$ and $\beta=1$, then $\L$ is the Whitehead link. 
    \item If $\ell=3$ and $\mu_{123}(\L)=\pm 1$, then $\L$ is the Borromean rings.
    \item If $\ell\geq 4$, then $\L$ is the four-component unlink.
    \end{enumerate}
\end{theorem}

Finally, if we drop the assumption that the components of $\L=L_1\cup L_2$ are unlinks, requiring only that it is an L-space link, we widen the detection result as follows.

\begin{theorem}
\label{whitehead or trefoil}
Let $\L=L_{1}\cup L_{2}$ be an algebraically split L-space link such that $S^{3}_{1, 1}(\L)$ is the Poincar\'e homology sphere. Then $\L$ is either the Whitehead link or the split union of $T(2, 3)$ and the unknot.
\end{theorem}

We make the following conjecture for three-component algebraically split L--space links. 

\begin{conjecture}
\label{Borromean}
Let $\L=L_{1}\cup L_{2}\cup L_{3}$ be an algebraically split L-space link such that $S^{3}_{1, 1, 1}(\L)$ is the Poincar\'e homology sphere. Then $\L$ must be one of the following:
\begin{enumerate}
\item the Borromean rings, 
\item the split union of the Whitehead link and the unknot, 
\item the split union of the right hand trefoil and the two-component unlink. 
\end{enumerate}
\end{conjecture}

\begin{remark}
By a similar argument to the one in \cite[Proposition 5.6]{GLM}, one of the manifolds $S^{3}_{1}(L_{i})$ is an L-space where $i=1, 2, 3$. Hence at least one of the components of $L$ is the unknot or the right-handed trefoil $T(2, 3)$. 
\end{remark}

Theorems~\ref{Lspace link detection combined} is sufficient to establish the following.   

\begin{corollary}\label{thm:cfl-detection}
The link Floer chain complex detects the Whitehead link and the Borromean rings.   
\end{corollary}
\begin{proof}
The link Floer chain complex determines the Heegaard Floer homology of large surgeries on a link and hence determines whether a link is an L-space link.  The link Floer chain complex also determines the multivariable Alexander polynomial \cite{OS:linkpoly}, which in turn determines the linking number \cite{Torres}, as well as whether the link is Brunnian. By Theorem~\ref{Lspace link detection combined}, it thus suffices to know that the link Floer complex determines the Sato-Levine invariant (for links with two components) or Milnor triple linking number (for links with three components).  This is shown in \cite{GLM} for links with two components and Theorem~\ref{thm:milnor} for links with three components.
\end{proof}

The final appearance of the triple linking number involves its relationship with another Heegaard Floer-theoretic invariant, the $d$-invariant of a homology three-sphere.  As is common, we write $d$ for the $d$-invariant when computing Heegaard Floer homology with $\mathbb{Z}_2$-coefficients.  We will use $d_{\kk}$ for the $d$-invariant when working with coefficients in a field $\kk$.  We show that the non-vanishing of the triple linking number for a link $\L$ gives an interesting restriction on the $d$-invariants of integer homology spheres obtained by Dehn surgeries along $\L$ in the three-sphere. We prove the following:

\begin{theorem}
\label{nonzero linking d invt}
Let $\L=L_1\cup L_2\cup L_3$ be an algebraically split link such that all two-component sublinks are $\Q$-L--space links. If the triple linking number $\mu_{123}$ is nonzero, then $d_{\Q}(S^3_{1,1,1} (\L))\leq -2$. If the triple linking number $\mu_{123}$ is odd, then the analogous inequality holds with $\mathbb{Z}_2$-coefficients.
\end{theorem}

In Theorem \ref{nonzero linking d invt}, we use $d$-invariants for Heegaard Floer homology with coefficients in $\Q$ and in $\mathbb{Z}_2$. Although it is customary for Heegaard Floer homology literature to work over $\mathbb{Z}_2$, we use a comparison with monopole Floer homology to obtain a separate result over $\Q$. Surprisingly enough, we do not know if $d_{\Q}$ and $d$ coincide. However, note that a $\mathbb{Z}_2$-L-space is necessarily a $\mathbb{Q}$-L-space.  See Section \ref{subsec:coefs} for a discussion about coefficients.

\begin{remark}
The same proof applies if we work with coefficients in an arbitrary field $\kk$ and the characteristic of $\kk$ is coprime to $\mu_{123}(\L)$.
\end{remark}

The L--space link assumption in Theorem~\ref{nonzero linking d invt} will be shown to be necessary in Example \ref{ex: bing}.  As an immediate application of Theorem \ref{nonzero linking d invt}, notice that when $\L=L_1\cup L_2\cup L_3$ is an algebraically split link such that all two-component sublinks are L--space links and has non-vanishing  triple linking number $\mu_{123}$, then $S^3_{1,1,1}(\L)$ has infinite order in homology cobordism group, and, for example, does not bound rational homology ball.

\begin{remark}
It is easy to see that Theorem \ref{nonzero linking d invt} holds for algebraically split links $\L$ with $n$ components which contains a three-component sublink $L_{ijk}$ satisfying the assumption of the theorem. 
\end{remark}

In Section \ref{sec:d} we establish some $d$-invariant inequalities for surgeries on links over an arbitrary field $\kk$ which may be of independent interest to the reader. Since link Floer homology is only defined over $\mathbb{Z}_2$ at the moment, for arbitrary coefficients, we cannot make use of it or various formulas relating link Floer homology with the Heegaard Floer homology of surgery.  Nonetheless, we obtain the following results (which are known to experts over $\mathbb{Z}_2$):

\begin{proposition}
\label{prop: L space sublink}
Assume that $\L$ is a nontrivial $\kk$-L--space link of any number of components and pairwise linking zero. Then $\dk(S^3_{1, \cdots, 1}(\L))\le -2$. The same inequality is true for any $(1/m_1,\ldots,1/m_\ell)$--surgery along $\L$ where $m_1,\ldots,m_\ell>0$.
\end{proposition}

 \begin{corollary}
\label{unlinksphere}
Let $\L $ be an algebraically split $\kk$-L-space link such that $S^{3}_{1, \cdots, 1}(\L)$ is $S^{3}$. Then $\L$ is the unlink.
\end{corollary}

\subsection*{Organization}  In Section~\ref{sec:background}, we review certain numerical invariants extracted from the link Floer complex, including (and defining) the $h$-function.  In Section~\ref{sec:milnor-h}, we relate these numerical invariants to the Milnor invariants and prove Theorem~\ref{thm:milnor}.  In Section~\ref{sec:detection} we establish the detection results for the Whitehead link and Borromean rings: Theorems~\ref{Brunnian detection from Poincare surgery combined}, ~\ref{Lspace link detection combined} and \ref{whitehead or trefoil}. Finally, in Section~\ref{sec:d}, we establish Theorem~\ref{nonzero linking d invt} on the $d$-invariants of surgeries on three-component algebraically split links.  

\subsection*{Acknowledgements}  We thank Marco Golla, Robert Lipshitz, Filip Misev, Yi Ni and Abigail Thompson for useful discussions.
E.G.~was supported by the grants DMS-1700814 and DMS-1760329.  T.L.~was supported by DMS-1709702 and a Sloan Fellowship.  B.L.~is grateful to Max Planck Institute for Mathematics in Bonn for its hospitality and financial support.  A.H.M.~was partially supported by DMS-1716987.  

%%%%%%%%%%%%%%%%%%%%%%%%%%%%%%%%%
%%%%%%%%%%%%%%%%%%%%%%%%%%%%%%%%%
%%%%%%%%%%%%%%%%%%%%%%%%%%%%%%%%%
%%%%%%%%%%%%%%%%%%%%%%%%%%%%%%%%%

\section{Background}
\label{sec:background}
In this section, we review the relevant aspects of Heegaard Floer homology, especially properties of the link Floer complex and L-space links. A multi-component link is denoted by scripted $\L$, and its components are denoted $L_i$. We denote multi-framings and vectors in an $n$-dimensional lattice by bold letters (e.g. $\bm{p}=(p_1, \cdots, p_n)$--surgery or $\bm{s}=(s_1, \cdots, s_n)$), and we let $\bm{e}_{i}$ denote a vector in $\Z^{n}$ where the $i$-th entry is $1$ and other entries are $0$. For any subset $B\subset\{1, \cdots, n\}$, we let $\bm{e}_B=\sum_{i\in B} \bm{e}_i$. Given two vectors $\bm{u}=(u_{1}, u_{2}, \cdots, u_{n})$ and $\bm{v}=(v_{1}, \cdots, v_{n})$ in $\mathbb{Z}^{n}$,  we write $\bm{u}\preceq \bm{v}$  if $u_{i}\leq v_{i}$ for each $1\leq i\leq n$, and $\bm{u}\prec \bm{v}$ if $\bm{u}\preceq \bm{v}$ and $\bm{u}\neq \bm{v}$.

\begin{definition}
\label{split}
A link $\L=L_1\cup \cdots \cup L_n$ is algebraically split if for all $i\neq j$, $\lk(L_i, L_j)=0$.
\end{definition}

Throughout this article, all links will be assumed to be algebraically split, unless otherwise stated.  However, we include a slightly more general discussion below for the benefit of the reader.

\begin{definition}
\label{Hfunctions}
For an oriented link $\L=L_{1}\cup \cdots \cup L_{n}\subset S^{3}$, define $\mathbb{H}(\L)$ to be the affine lattice over $\Z^{n}$,
\[
	\H(\L)=\oplus_{i=1}^{n}\H_{i}(\L), \quad \H_{i}(\L)=\Z+\dfrac{\lk(L_{i}, \L\setminus L_{i})}{2}.
\]
If $\L$ is algebraically split then $\H(\L)=\Z^n$.
\end{definition} 

\subsection{$d$-invariants}
\label{subsec:standard}
We assume familiarity with Heegaard Floer homology, and refer the reader to \cite{OS03, MO} for details. 
With the exception of Section \ref{sec:d}, we work over the field $\F=\Z_2$, as is typical in Heegaard Floer homology. 
Recall the $d$-invariant $d(Y, \mft$) of a rational homology sphere $Y$ equipped with a $\spinc$ structure $\mft$ is defined to be the maximal degree of a non-torsion class $x\in HF^{-}(Y, \mft)$. 
For three-manifolds with $b_1(Y)>0$, the definition of the $d$-invariant is more complicated, see Section \ref{subsec:d std}.

\subsection{The $h$-function and L--space links}
\label{sec:hfunction}

We review the definition of the $h$-function for oriented links $\L\subset S^{3}$, as defined by the first author and N\'emethi \cite{GN}. We will quote without proof several technical lemmas regarding its properties; proofs of these statements can be found in either \cite{BG}, \cite{GLM}, or both. 

A link $\L=L_{1}\cup \cdots \cup L_{n}$ in $S^{3}$ defines a filtration on the Floer complex $CF^{-}(S^{3})$. This filtration is indexed by elements $\bm{s}=(s_{1}, \cdots, s_{n})\in \H(\L)$ (see Definition \ref{Hfunctions}). 
The \emph{generalized Heegaard Floer complex} $\A^{-}(\L, \bm{s}) \subset CF^{-}(S^3)$ is the $\F[U]$-module defined to be the subcomplex of $CF^{-}(S^{3})$ corresponding to the filtration indexed by the lattice point $\bm{s}$ \cite{MO}. The large surgery theorem of \cite[Theorem 12.1]{MO} implies that the homology of $\A^{-}(\L, \bm{s})$ is isomorphic to the Heegaard Floer homology of a sufficiently large surgery on the link $\L$ equipped with some Spin$^{c}$-structure as an $\F[U]$-module. Therefore, there is a non-canonical isomorphism between the homology of $\A^{-}(\L, \bm{s})$ and a direct sum of one copy of $\F[U]$ and a $U$-torsion submodule. Thus the following is well-defined:

\begin{definition}\cite[Definition 3.9]{BG}
\label{Hfunction}
For an oriented link $\L\subseteq S^{3}$, we define the $H$-function $H_{\L}(\bm{s})$ by saying that $-2H_{\L}(\bm{s})$ is the maximal homological degree of a nonzero element in the free part of $H_{\ast}(\A^{-}(\L, \bm{s}))$ where $\bm{s}\in \H(\L)$. 
\end{definition}

\begin{remark}
We will write $H_{\L}(\bm{s})$ as $H(\bm{s})$ for brevity if the context is clear. The notation $H_I(\bm{s})$ refers to the $H$-function of the sublink determined by indices $I\subseteq\{1, \cdots, n\}$.   
\end{remark}

By definition $-2H_{\L}(\bm{s})$ is equivalent to the $d$-invariant of large surgery on $\L$, following a degree shift that depends only on the linking matrix and $\bm{s}$ (this is explained in \cite[Section 10]{MO}, \cite[Theorem 4.10]{BG}). In particular, the $H$-function is a well-defined topological invariant of links. 
For a knot $K$, the $H$-function satisfies $H_K(s) = V_s$, where $V_s$ are the similarly defined integer-valued concordance invariants coming from the knot Floer complex \cite{Ras:Thesis, NiWu}. 

We will work with a normalized version of the $H$-function, denoted $h$, as well as a `state sum', denoted $h'$, that is taken over sublinks of $\L$. 
\begin{definition}
Let $\L$ be an $n$--component algebraically split link, $n\ge 1$. We define
\begin{align*}
h(\bm{s})&=H(\bm{s})-H_O(s_1)-\ldots-H_O(s_n) \\
h'(\bm{s})&=\sum_{I\subseteq \{1,\ldots,n\} }(-1)^{n-|I|}h_{I}(\bm{s})
\end{align*}
where $H_O(s)$ is the $H$-function for the unknot and $h(\emptyset)=0$.  
\end{definition}

\begin{example}
\label{ex:hbrun}
Assume that $\L$ is a Brunnian link, that is, all proper sublinks are unlinks. Then $h_I(\bm{s})=0$ for all proper subsets $I$ and
$
h'(\bm{s})=h(\bm{s}).
$
\end{example}

We now list several properties of the $H$-function. 

\begin{lemma}
\label{h-function list} 
For an oriented link $\L\subseteq S^{3}$, 
\begin{enumerate}
	\item 
	\label{h-function nonnegative}
	The $H$-function $H_{\L}(\bm{s})$ takes nonnegative values.
	\item 
	\label{h-function increase}
	$H_{\L}(\bm{s}-\bm{e}_{i})=H_{\L}(\bm{s})$ or $H_{\L}(\bm{s}-\bm{e}_{i})=H_{\L}(\bm{s})+1$ where $\bm{s}\in \H$. 
	\item 
	\label{h-function symmetry}
	$H(-\bm{s})=H(\bm{s})+\sum_{i=1}^{n} s_i$ where $\bm{s}=(s_1, \cdots, s_n)$.
	\item 
	\label{h-function bdy}
	If $\L$ is algebraically split and $N$ is sufficiently large, then
    \[
    	H_{\L}(s_1, \cdots, s_{i-1}, N, s_{i+1},  \cdots, s_n)=H_{\L\setminus L_i}(s_1, \cdots, s_{i-1}, s_{i+1}, \cdots, s_n),
    \] 
	for all $i$ and $s_j$.  
    \item 
    \label{h-symmetry}
    The functions $h$ and $h'$ have the following symmetry property:
    \[
    h(-\bm{s})=h(\bm{s}),\ h'(-\bm{s})=h'(\bm{s}).
    \]
    \item 
    \label{lem: h increases}
    The function $h$ is non-decreasing towards the origin. That is, $h(\bm{s}-\bm{e}_i)\ge h(\bm{s})$ if $s_i>0$ and $h(\bm{s}-\bm{e}_i)\le h(\bm{s})$ if $s_i\le 0$.
    \item 
    \label{h nonnegative}
	For all $\bm{s}$ one has $h(\bm{s})\ge 0$.
\end{enumerate}
\end{lemma}

\begin{proof}
Items \eqref{h-function nonnegative} and \eqref{h-function increase} are proved in \cite[Proposition 3.10]{BG} and \eqref{h-function symmetry} is proved in \cite[Lemma 5.5]{LiuY2}. Item \eqref{h-function bdy} is \cite[Proposition 3.12]{BG}. For \eqref{h-symmetry}, the first equation follows from \eqref{h-function symmetry} and the second follows from the first.  Item \eqref{lem: h increases} and Item \eqref{h nonnegative} are proved in \cite[Lemma 2.16, Corollary 2.17]{GLM}. 
\end{proof}

\begin{lemma}
\label{lem: h' finite}
Let $\L$ be an algebraically split link. Then
the function   $h'_{\L}(\bm{s})$ is finitely supported. 
\end{lemma}

\begin{proof}
By Lemma \ref{h-function list}\eqref{h-function bdy}, when $s_i>N$ for some $N$, we have $h_{I\sqcup\{i\}}(\bm{s})=h_{I}(\bm{s})$. 
By Lemma \ref{h-function list}\eqref{h-symmetry}, it is therefore sufficient to prove for all $i$ that when $s_i\to +\infty$ we have $h'(\bm{s})=0$. 
Fix any index $i$ and observe that we can write
\[
h'(\bm{s})=\sum_{i \notin I }(-1)^{n-|I|}(h_{I}(\bm{s})-h_{I\sqcup \{i\}}(\bm{s})).\qedhere
\]
\end{proof}

\begin{lemma}
\label{split trivial}
Suppose that $\L$ is an $n$--component split link and $n>1$. Then $h'(\bm{s})=0$ for all $\bm{s}$.
\end{lemma}

\begin{proof}
Since $\L$ is split, we have $h_{I}(s)=\sum_{i\in I}h_i(s_i)$ for all $I$, and
$$
h'(\bm{s})=\sum_{I}(-1)^{n-|I|}\sum_{i\in I}h_i(s_i)=\sum_{i}h_i(s_i)\sum_{i\in I}(-1)^{n-|I|}=0
$$
for $n>1$. 
\end{proof}

\begin{corollary}
\label{cor: h' from H}
Let $\L$ be an algebraically split link with $n>1$ components. Then
$$
h'(s_1,\ldots,s_n)=\sum_{I\subseteq \{1,\ldots,n\} }(-1)^{n-|I|}H_{I}(\bm{s}).
$$
\end{corollary}

\begin{proof}
If $\L$ is an unlink, then similarly to Lemma \ref{split trivial} the right hand side vanishes, and then the statement follows by linearity.
\end{proof}

\begin{lemma}
\label{lem: alternating h'}
Let $\L$ be an algebraically split link with $n>1$ components. Then for all $\bm{s}$
one has
$$
\sum_{J\subset \{1,\ldots,n\}}(-1)^{|J|}h'_{\L}(\bm{s}-\bm{e}_{J})=\sum_{J\subset \{1,\ldots,n\}}(-1)^{|J|}H_{\L}(\bm{s}-\bm{e}_{J}).
$$
\end{lemma}

\begin{proof}
By Corollary \ref{cor: h' from H} we get
$$
\sum_{J\subset \{1,\ldots,n\} }(-1)^{|J|} h'(\bm{s}-\bm{e}_{J})=
\sum_{I, J\subset \{1,\ldots,n\} }(-1)^{n-|I|+|J|} H_{I}(\bm{s}-\bm{e}_{J})
$$
If $I$ is a proper subset of $\{1\,\ldots,n\}$ then $H_{I}(\bm{s}-\bm{e}_{J})=H_{I}(\bm{s}-\bm{e}_{I\cap J})$ and we can write $J=J'\sqcup J''$ where $J'=I\cap J$ and $J''=J\setminus I$. Then
$$
\sum_{J\subset \{1,\ldots,n\}}(-1)^{n-|I|+|J|}H_{I}(\bm{s}-\bm{e}_{J})=\sum_{J'\subset I}(-1)^{n-|I|+|J'|}H_{I}(\bm{s}-\bm{e}_{J'})\sum_{J''\subset (\{1,\ldots,n\}\setminus I)}(-1)^{|J''|}=0,
$$
so the only surviving terms are from $I=\{1,\ldots,n\}$. 
\end{proof}

\subsection{L--spaces}
Recall from \cite{OS:lenspaces} that a rational homology sphere $Y$ is an L--space if it is has the simplest possible Heegaard Floer homology. More precisely, for any Spin$^{c}$-structure $\mathfrak{s}$, $HF^{-}(Y, \mathfrak{s})$ is a free $\F[U]$-module of rank one.

\begin{definition}\cite{GN,LiuY2}
\label{definition of L--space link}
An oriented $n$-component link $\L\subset S^{3}$ is an L--space link if there exists  $\bm{0}\prec \bm{p}\in \mathbb{Z}^{n}$ such that the surgered manifold $S^{3}_{\bm{q}}(\L)$ is an L--space for any $\bm{q}\succeq \bm{p}$. 
\end{definition}

Recall that if a knot $K\subset S^3$ admits any positive surgery to an L--space, then $S^3_{p/q}(K)$ is also an L--space for all $p/q\ge 2g(K)-1$ \cite{OS:lenspaces}. For links though, it is not necessarily the case that the existence of a single $\bm{p}$--surgery yielding an L--space guarantees that all large surgeries are also L--spaces. However, the following criterion of Y. Liu can determine when this is the case. 
\begin{theorem}\cite{LiuY2}
\label{l-space link cond}
\begin{enumerate}
	\item \label{sublink} Every sublink of an L--space link is an L--space link.
    \item \label{no torsion} A link is an L--space link if and only if  for all $\bm{s}$ one has $H_{\ast}(\A^{-}(\L, \bm{s}))=\F[U]$.
	\item \label{criterion} Assume that for some $\bm{p}$ the surgery $S^{3}_{\bm{p}}(L)$ is an L--space. In addition, assume that for all sublinks $\L'\subset \L$ the surgeries $S^{3}_{\bm{p}|_{\L'}}(\L')$ are L--spaces too, and the framing matrix $\Lambda|_{\L'}$ is positive definite.
Then for all $\bm{q}\succeq \bm{p}$ the surgered manifolds $S^{3}_{\bm{q}}(\L)$ are L--spaces, and so $\L$ is an L--space link.
\end{enumerate}
\end{theorem}

\begin{example}
\label{ex:Lspacelinks}
If $\L$ is algebraically split, then $\Lambda$ is positive definite if and only if all $p_i>0$. In this case, the existence 
of any $\bm{p}$ with all $p_i > 0$ such that $S^{3}_{\bm{p}'}(\L')$ is an L--space for all sublinks $\L'$ is sufficient to conclude $\L$ is an L--space link. In particular, any Brunnian link admitting a single positive L--space surgery (e.g. $(+1,+1,+1)$-surgery along the Borromean rings) is an L--space link.
\end{example}

By \cite{OS:linkpoly}, the Euler characteristic $\chi(HFL^{-}(\L, \bm{s}))$ is the multivariable Alexander polynomial,
\begin{equation}
\label{computation 3}
\tilde{\Delta}_{\L}(t_{1}, \cdots, t_{n})=\sum_{\bm{s}\in \H(\L)}\chi(HFL^{-}(\L, \bm{s}))t_{1}^{s_{1}}\cdots t_{n}^{s_{n}}
\end{equation}
where $\bm{s}=(s_{1}, \cdots, s_{n})$, and
\begin{equation}
\label{mva}
\widetilde{\Delta}_{\L}(t_{1}, \cdots, t_{n}): = \left\{
        \begin{array}{ll}
           (t_{1}\cdots t_{n})^{1/2} \Delta_{\L}(t_{1}, \cdots, t_{n}) & \quad \textup{if } n >1, \\
            \Delta_{\L}(t)/(1-t^{-1}) & \quad  \textup{if } n=1. 
        \end{array}
    \right. 
\end{equation}
For L--space links, the $H$-function can be computed from the multi-variable Alexander polynomial.
Indeed, by Theorem \ref{l-space link cond} \eqref{no torsion} and the inclusion-exclusion formula, one can write
\begin{equation}
\label{computation of h-function 1}
\chi(HFL^{-}(\L, \bm{s}))=\sum_{B\subset \lbrace 1, \cdots, n \rbrace}(-1)^{|B|-1}H_{\L}(\bm{s}-\bm{e}_{B}),
\end{equation}
as in \cite[(3.14)]{BG}. 

\begin{example}
\label{wh H}
The (symmetric) Alexander polynomial of the Whitehead link equals 
\[
\Delta(t_1,t_2)=-(t_1^{1/2}-t_1^{-1/2})(t_2^{1/2}-t_2^{-1/2}),
\]
and the $H$-function has the following values.

\begin{center}
\begin{tikzpicture}
\draw (1,0)--(1,5);
\draw (2,0)--(2,5);
\draw (3,0)--(3,5);
\draw (4,0)--(4,5);
\draw (0,1)--(5,1);
\draw (0,2)--(5,2);
\draw (0,3)--(5,3);
\draw (0,4)--(5,4);
\draw (0.5,4.5) node {2};
\draw (1.5,4.5) node {1};
\draw (2.5,4.5) node {0};
\draw (3.5,4.5) node {0};
\draw (4.5,4.5) node {0};
\draw (0.5,3.5) node {2};
\draw (1.5,3.5) node {1};
\draw (2.5,3.5) node {0};
\draw (3.5,3.5) node {0};
\draw (4.5,3.5) node {0};
\draw (0.5,2.5) node {2};
\draw (1.5,2.5) node {1};
\draw (2.5,2.5) node {1};
\draw (3.5,2.5) node {0};
\draw (4.5,2.5) node {0};
\draw (0.5,1.5) node {3};
\draw (1.5,1.5) node {2};
\draw (2.5,1.5) node {1};
\draw (3.5,1.5) node {1};
\draw (4.5,1.5) node {1};
\draw (0.5,0.5) node {4};
\draw (1.5,0.5) node {3};
\draw (2.5,0.5) node {2};
\draw (3.5,0.5) node {2};
\draw (4.5,0.5) node {2};
\draw [->,dotted] (0,2.5)--(5,2.5);
\draw [->,dotted] (2.5,0)--(2.5,5);
\draw (5,2.7) node {$s_1$};
\draw (2.3,5) node {$s_2$};
\end{tikzpicture}
\end{center}

The $H$-function of the two-component unlink agrees everywhere with the $H$-function of the Whitehead link except at $\bm{s}=(0,0)$, where $H_{\bm{O}}(\bm{0})=0$. Therefore, for the Whitehead link,
\begin{equation}
\label{whitehead h}
h_{\L}(s_1,s_2)=\begin{cases}
1 & \text{if}\ s_1=s_2=0\\
0 & \text{otherwise}.
\end{cases}
\end{equation}

\end{example}
Lastly, we observe:
\begin{lemma}
\label{unlink1}
If for an L--space link $\L$ one has $h(\bm{0})=0$,  then $\L$ is the unlink. 
\end{lemma}

\begin{proof}
If $h(\bm{0})=0$ then by Lemma \ref{h-function list} \eqref{lem: h increases} we have $h(\bm{s})=0$ for all $\bm{s}\in \H(\L)$. The rest of the proof follows from \cite[Theorem 1.3]{Beibei}.
\end{proof}

We will also make use of the following well-known fact without reference.  If $K$ is an L--space knot, then 
\[
g(K) = \max \{s \mid h(s) > 0 \} + 1.
\]

\subsection{Coefficients}
\label{subsec:coefs}

As stated above, for most of the paper we use $\F=\mathbb{Z}_2$ as the field of coefficients.
However, in Section \ref{sec:d} we will use rational coefficients, so we need to discuss the dependence of the results on the field of coefficients. 

First of all, the Heegaard Floer complexes $\CFhat,\CFminus,\CFinfty$ for knots and three-manifolds are defined over $\Z$ \cite{OS03}. In particular, $\CFminus$ is a complex of finitely generated free $\Z[U]$ modules. Its decomposition into $\spinc$ structures is well defined over $\Z[U]$.

Let $\kk$ be an arbitrary field. We will write $\CFminus_{\kk}=\CFminus\otimes_{\Z}\kk$, and define
$\CFhat_{\kk},\CFinfty_{\kk}$ and $\HFhat_{\kk},\HFminus_{\kk},\HFinfty_{\kk}$  similarly. Since $\kk[U]$ is a principal ideal domain, any graded $\kk[U]$ module (in particular, $\HFminus_{\kk}$) can be decomposed as a direct sum of several copies of $\kk[U]$ and $\kk[U]/U^{d_i}$ for various
$d_i$.   

If $Y$ is a rational homology sphere, then for any $\spinc$ structure $\mft$ on $Y$ and any field $\kk$
the homology $\HFminus_{\kk}(Y,\mft)$ contains exactly one copy of $\kk[U]$ \cite{OS:properties}. We define $d_{\kk}(Y,\mft)$, the $d$-invariant with coefficients in $\kk$, as the homological degree of the generator of this copy of $\kk[U]$. When $\kk=\F=\Z_2$, we simply write $d(Y,\mft)=d_{\F}(Y,\mft)$, as above.

The following two examples show that $d$-invariants with coefficients in $\F$ and in $\Q$ could be potentially very different. It would be very interesting (but rather challenging) to find such examples 
in actual Heegaard Floer homology. In both examples we consider complexes of free $\Z[U]$-modules with three generators $a,b,c$. 

\begin{example} 
Suppose that  that $\partial(c)=U^{k}a-2b$.
The homology over $\Z$ can be identified with the submodule of $\Z[U]$ generated by 
2 (corresponding to $a$) and $U^k$ (corresponding to $b$). In particular, the homology is free as 
$\mathbb{Z}$-module and has no torsion.

On the other hand, if we consider this complex over $\F$, then $\partial(c)=U^{k}a$ and 
the homology is isomorphic to $\F[U]/(U^k)\oplus\F[U]$ as graded $\F[U]$-module,  its $\F[U]$ free part is generated by $b$. If we consider the same complex over $\Q$, then it is isomorphic to $\Q[U]$ generated by $a$. 

In conclusion, $d_{\F}=d_{\Q}-2k.$
\end{example}

\begin{example}
Suppose  that $\partial(a)=U^{k}c$ and $\partial(b)=2c$. In this case the homology has $\Z_2$ torsion of rank $k$, spanned by $c,Uc,\ldots,U^{k-1}c$. 

If we consider this complex over $\F$, then the homology is isomorphic to $\F[U]/(U^k)\oplus\F[U]$ as graded $\F[U]$-module, and its $\F[U]$ free part is generated by $b$. If we consider the same complex over $\Q$, then it is isomorphic to $\Q[U]$ generated by $2a-U^{k}b$.

In conclusion, $d_{\F}=d_{\Q}+2k.$
\end{example}
 
Note that the above examples show that the difference between $d_{\F}$ and $d_{\Q}$ could be either positive or negative, and arbitrarily large in absolute value. 

It is also important to point out that
that the notion of L-space (and hence of L-space link) depends on the coefficients, so a pedantically inclined reader is invited to use the terms $\F$-L-space and $\F$-L-space link.  

At present, link Floer homology (for links with more than one component) is only defined over $\F$, so the $H$-function and its cousins are only defined over $\F$.

%%%%%%%%%%%%%%%%%%%%%%%%%%%%%%%%%
%%%%%%%%%%%%%%%%%%%%%%%%%%%%%%%%%
%%%%%%%%%%%%%%%%%%%%%%%%%%%%%%%%%
%%%%%%%%%%%%%%%%%%%%%%%%%%%%%%%%%

\section{Milnor invariants and the Casson invariant}\label{sec:milnor-h}
In this section, we show how to extract the Milnor triple linking invariant from the link Floer complex.  This will be in terms of the invariant $h'$ defined in the previous section and another invariant $\chi'$  from the torsion part of $H_{\ast}(\A^{-}(\L, \bm{s}))$.  

\subsection{The invariant $\chi'$}

For non L--space links, the $h$-function does not determine the Alexander polynomial.  However, we can obtain this from the collection of $H_*(\A^-(\L,\bm{s}))$ for all $\bm{s}$, which we now explain.  Recall that for any link we have a non-canonical splitting 
$$
H_*(\A^-(\L,\bm{s}))=\F[U][-2H(\bm{s})]\oplus \A^-_{tor}(\L,\bm{s}),
$$
where $\A^-_{tor}(\L,\bm{s})$ is finite-dimensional over $\F$ and hence a torsion module over $\F[U]$.  We begin by analyzing the modules $\A^-_{tor}$, as they will feature in our formula for the Alexander polynomial, and ultimately the Milnor invariants.

\begin{lemma}
For an algebraically split link $\L$, we have
$$
H_*(\A^-_{tor}(\L,-\bm{s})) \cong H_*(\A^-_{tor}(\L,\bm{s}))[-2|\mathbf{s}|],
$$  
where $|\mathbf{s}| = \sum_i s_i$.
\end{lemma}

\begin{proof}
By the large surgery theorem \cite{MO} we have (up to a grading shift)
\[
H_*(\A^-(\L,\bm{s})) \cong HF^-(S^3_{\bm{p}}(\L),\bm{s}),\ \text{for}\ \bm{p}\ggcurly 0.
\]
By \cite[Theorem 2.4]{OS:properties} we have 
\[
HF^-(S^3_{\bm{p}}(\L),-\bm{s}) \cong HF^-(S^3_{\bm{p}}(\L),\bm{s})
\]
Therefore up to a grading shift we have
\[
H_*(\A^-(\L,-\bm{s})) \cong H_*(\A^-(\L,\bm{s})).
\]
To figure  out the shift, we can look at the $\F[U]$-free part and use the identity
\[
H(-\bm{s})=H(\bm{s})+|\bm{s}|.\qedhere
\]\end{proof}

\begin{corollary}
For an algebraically split link $\L$, we have 
\begin{equation}
\label{eq: tor symmetry}
\chi(\A^-_{tor}(\L,-\bm{s}))=\chi(\A^-_{tor}(\L,\bm{s})).
\end{equation}
\end{corollary}

We are ready to define the analogue of the function $h'$ for the torsion parts of $\A^-$.

\begin{definition}
Let $\L$ be an algebraically split link with $n$ components. We define
\[
\chi'_{\L}(\bm{s})=\chi'(\bm{s})=\sum_{I\subset \{1,\ldots,n\}}(-1)^{n-|I|}\chi(\A^-_{tor}(\L_{I},\bm{s}_{I})).
\]
\end{definition}

\begin{lemma}
\label{lem: chi' finite}
The function $\chi'(\bm{s})$ is finitely supported and enjoys the symmetry $\chi'(-\bm{s})=\chi'(\bm{s})$.
\end{lemma}

\begin{proof}
The symmetry for $\chi'(\bm{s})$ immediately follows from \eqref{eq: tor symmetry}. Let us prove that it is finitely supported.
For $s_i\gg 0$ and any subset $I$ not containing $s_i$ we have from \cite[Lemma 10.1]{MO}
\[
H_*(\A^-(\L_{I},\bm{s}_{I})) \cong H_*(\A^-(\L_{I\cup \{i\}},\bm{s}_{I\cup\{i\}}))
\]
and, in particular, 
\[
\chi(\A^-_{tor}(\bm{s}_{I}))=\chi(\A^-_{tor}(\bm{s}_{I\cup\{i\}})).
\]
Then similarly to Lemma \ref{lem: h' finite} we have $\chi'(\bm{s})=0$ for $s_i\gg 0$. By symmetry, we also have
$\chi'(\bm{s})=0$ for $s_i\ll 0$, and therefore $\chi'(\bm{s})$ is finitely supported. 
\end{proof}

Finally, equation \eqref{computation of h-function 1} can be generalized in the presence of torsion as follows:

\begin{equation}
\label{computation of h-function tors}
\chi(HFL^{-}(\L, \bm{s}))=\sum_{B\subset \lbrace 1, \cdots, n \rbrace}(-1)^{|B|}(\chi(\A^{-}_{tor}(\bm{s}-\bm{e}_B))-H_{\L}(\bm{s}-\bm{e}_{B})).
\end{equation}

Let $\L$ be an algebraically split link with $n>1$ components. Then the Torres condition \cite{Torres}
implies that $\Delta_{\L}$ is divisible by $(t_i-1)$ for all $i$. Hence, we can write
\begin{equation}
\label{eq: def Delta'}
\Delta_{\L}(t_1,\ldots, t_n) =\prod_i (t_i^{1/2}-t_i^{-1/2})\widetilde{\Delta}'_\L(t_1, \cdots, t_n),
\end{equation}
where $\Delta_\L$ is normalized as in equation \eqref{mva} above.

\begin{theorem}
\label{lem: Delta'}
Assume that $\L$ is an algebraically split link with $n>1$ components. Then 
$$
\tilde{\Delta}'_{\L}(\bm{t})=(-1)^{n}\sum_{\bm{s}}(\chi'(\bm{s})-h'(\bm{s}))\bm{t}^{\bm{s}},
$$
where $ \bm{t}^{\bm{s}}=t_1^{s_1}\cdots t_n^{s_n}$.
\end{theorem}

\begin{proof}
Let $\tilde{\Delta}'_{\L}(\bm{t})=\sum q(\bm{s})\bm{t}^{\bm{s}}$ and $\tilde{\Delta}_{\L}(\bm{t})=\prod_{i}(t_{i}-1)\tilde{\Delta}'_{\L}(t)=\sum a(\bm{s})\bm{t}^{\bm{s}}.$
Then
\begin{equation}
\label{a from q}
a(\bm{s})=\sum_{J\subset \{1,\ldots,n\}}(-1)^{n-|J|}q(\bm{s}-\bm{e}_{J}).
\end{equation}
By Lemma \ref{lem: alternating h'} we get 
$$
\sum_{J\subset \{1,\ldots,n\}}(-1)^{|J|}h'_{\L}(\bm{s}-\bm{e}_{J})=\sum_{J\subset \{1,\ldots,n\}}(-1)^{|J|}H_{\L}(\bm{s}-\bm{e}_{J}).
$$
Similarly,
$$
\sum_{J\subset \{1,\ldots,n\}}(-1)^{|J|}\chi'(\bm{s}-\bm{e}_{J})=\sum_{J\subset \{1,\ldots,n\}}(-1)^{|J|}\chi(\A^-_{tor}(\bm{s}-\bm{e}_{J})).
$$
Therefore \eqref{computation of h-function tors} implies 
\begin{equation}
\label{a from h'}
a(\bm{s})=\sum_{J\subset \{1,\ldots,n\}}(-1)^{|J|}(\chi'(\bm{s}-\bm{e}_{J})-h'(\bm{s}-\bm{e}_{J})).
\end{equation}
Equations \eqref{a from q} and \eqref{a from h'} imply that $(-1)^{n}q(\bm{s})$ and $\chi'(\bm{s})-h'(\bm{s})$ satisfy the same recursion relations, and by Lemmas \ref{lem: h' finite}
and \ref{lem: chi' finite} both vanish for sufficiently large $\bm{s}$. Therefore, 
\[
\tilde{\Delta}'_{\L}(\bm{t})=(-1)^{n}\sum_{\bm{s}}(\chi'(\bm{s})-h'(\bm{s}))\bm{t}^{\bm{s}}.\qedhere
\]
\end{proof}

\begin{corollary}
\label{cor: Delta' Brunnian}
Assume that $\L$ is a Brunnian L--space link with $n>1$ components. Then 
$$ 
\tilde{\Delta}'_{\L}(\bm{t})=(-1)^{n+1}\sum_{\bm{s}}h(\bm{s})\bm{t}^{\bm{s}}.
$$
\end{corollary}

\begin{proof}
For Brunnian links $h'(\bm{s})=h(\bm{s})$, and since $\L$ is an L-space link, $\chi'(\bm{s})=0$ for all $\bm{s}$.
\end{proof}

\subsection{Milnor triple linking invariant}
To associate the Milnor triple linking invariant to the Alexander polynomial, we must pass through the Conway polynomial. 
The Conway polynomial of $\L=L_1\cup \ldots \cup L_n$ can be written as
\[
	\nabla_\L (z) = z^{n-1}(a_0 + a_2z^2 + a_4z^4 + \cdots ), \qquad a_i\in\Z,
\]
and 
\begin{equation*}
	\nabla_\L(t^{1/2}-t^{-1/2})=\pm (t^{1/2}-t^{-1/2})\Delta_{\L}(t,\ldots, t),
\end{equation*}
where $\Delta_{\L}(t_1,\ldots, t_n)$ denotes the multi-variable Alexander polynomial of $\L$. Note that $a_{0}$ depends only on  the linking numbers of $\L$, see \cite[Theorem 1]{Hoste:Conway}. 
For an algebraically split link with $n>1$ components, $a_{0}=0$ and we can write its multivariable Alexander polynomial as in \eqref{eq: def Delta'}:
\[
\Delta_{\L}(t_1,\ldots, t_n) =\prod_i (t_i^{1/2}-t_i^{-1/2})\widetilde{\Delta}'_\L(t_1, \cdots, t_n).
\]
Then we can set all $t_i=t$ and apply the change of variable $z=t^{1/2} - t^{-1/2}$:
\[
\nabla_\L(z)= \pm (t^{1/2} - t^{-1/2})\Delta_\L(t,\ldots,t) =
\pm  (t^{1/2} - t^{-1/2})^{n+1} \widetilde{\Delta}'_\L(t,\ldots, t).
\]
We define $\widetilde{\nabla}_\L(z)$ as $\nabla_\L(z) / z^{n+1}$.  With this, the  coefficient $a_{2}$ of the Conway polynomial can be written as
$$
a_2(\L)=\widetilde{\nabla}_\L(0)=\pm \tilde{\Delta}'_\L(1,\ldots, 1). 
$$
It is an important invariant of the link. By Theorem \ref{lem: Delta'} we get
\begin{equation}\label{eq:a2-chi}
a_2(\L)=\pm \sum_{\bm{s}}(\chi'(\bm{s})-h'(\bm{s})).
\end{equation}

\begin{example}
\label{ex: 2comp}
For two-component algebraically split links the invariant $a_2(\L)$ agrees with the Sato-Levine invariant $\beta(\L)$ up to sign \cite{GLM, Sturm}.  Equation \eqref{eq:a2-chi} now gives an explicit formula for $\beta(\L)$ in terms of the link Floer complex for $\L$. Moreover, if $\L$ is a two-component algebraically split  L-space link with unknotted components, then $\beta(\L)=0$ implies $\L$ is the unlink (see \cite[Corollary 6.4]{GLM}).
\end{example}

\begin{remark}
For two-component links, other Milnor invariants of the form $\mu[1^p 2^q]$ may be written in terms of the link Floer complex as follows. (For example, the linking number corresponds with $\mu[12]$ and the Sato-Levine invariant with $\mu[1122]$.) One first writes the multivariable Alexander polynomial as in Theorem \ref{lem: Delta'}. Then passing to the Taylor expansion at $(1, 1)$ of this two variable polynomial, a result of Murasugi \cite[Theorem 4.1]{Murasugi:Milnor} shows that the coefficients of the Taylor expansion determine these Milnor invariants. 
\end{remark}

\begin{example}
\label{ex: 4comp}
If $\L$ is an algebraically split link with $n\ge 4$ components then by \cite{Cochran}
$a_2(\L)=0.$
\end{example}

\begin{proof}[Proof of Theorem~\ref{thm:milnor}]
For three-component links, by  \cite[Theorem 5.1]{Cochran} $a_2(\L)$ relates to the square of the Milnor triple linking number
\begin{equation}\label{eq:a2-triple-linking}
	a_2(\L) =\pm \mu^2_{123}(\L).
\end{equation}
Equation \eqref{eq:a2-chi} now establishes an explicit formula for $|\mu_{123}|$ in terms of the link Floer complex.
\end{proof}

\begin{corollary}
\label{cor:mucfl}
Assume that $\L$ is an algebraically split link with three components. Then 
$$\mu^{2}_{123}(\L)=\pm \sum_{\bm{s}}(\chi'(\bm{s})-h'(\bm{s})).$$
In particular, $|\sum_{\bm{s}}(\chi'(\bm{s})-h'(\bm{s}))|$ is a link homotopy invariant.
\end{corollary}

For Brunnian L--space links we can get even more information.
\begin{theorem}
\label{Brunnianh}
Assume that $\L$ is a Brunnian L--space link with three components. Then the following statements hold:
\begin{enumerate}
\item[(a)] $\mu_{123}^2(\L)=\sum_{\bm{s}}h(\bm{s})$.
\item[(b)] If $\mu_{123}(\L)=0$ then $\L$ is the unlink.
\item[(c)] $\mu_{123}(\L)$ has the same parity as $h(0,0,0)=H(0,0,0)$.
\item[(d)] $\mu_{123}(\L)$ cannot equal $\pm 2$.
\end{enumerate}
\end{theorem}

\begin{proof}
(a) By \eqref{eq:a2-triple-linking} we have $\mu^2_{123}(\L)=\pm a_2(\L)$ and by 
Corollary \ref{cor:mucfl}  we have $a_2(\L)=\sum_{\bm{s}}h(\bm{s}).$ Therefore
$$
\mu^2_{123}(\L)=\pm \sum_{\bm{s}}h(\bm{s}).
$$
On the other hand, both $\mu^2_{123}$ and $h(\bm{s})$ are nonnegative, so the sign is positive.

(b) If $\mu_{123}(\L)=0$ then by (a) we have $h(\bm{s})=0$ for all $\bm{s}$. By \cite[Theorem 1.3]{Beibei} this implies that $\L$ is the unlink. 

(c) By Lemma \ref{h-function list}\eqref{h-symmetry}, we have $h(-\bm{s})=h(\bm{s})$ for all $\bm{s}$. Note that   $\bm{s}= -\bm{s}$ if and only if $\bm{s}= (0,0,0)$. Therefore
$\sum_{\bm{s}}h(\bm{s})$ has the same parity as $h(0,0,0)$.

(d) Assume that $\mu_{123}(\L)=2$, then by (c)  $h(0,0,0)$ is even as well. If $h(0,0,0)=0$ then $h(\bm{s})=0$ for all $\bm{s}$ by Lemma \ref{h-function list}\eqref{lem: h increases}, which is a contradiction. Therefore $h(0,0,0)\ge 2$, hence
$$
h( 1,0,0), h(0, 1,0), h(0,0, 1)\ge 1
$$
by Lemma \ref{h-function list}.  This implies $\mu^2_{123}=\sum_{\bm{s}}h(\bm{s})\ge 2+3=5,$ contradicting to the assumption.
\end{proof}

Note that Theorem \ref{Brunnianh}(a) and (b) provide the statement of Theorem \ref{triple squared} and Corollary \ref{unlink detection} in the introduction. We mention here an open problem suggested by the discussion above:
\begin{problem}
Do there exist examples of Brunnian L--space links with three components and large $\mu_{123}$?
\end{problem}

Let $\L=L_{1}\cup L_{2}\cdots \cup L_{n}$ be an oriented link in an integer homology sphere $Y$ with all pairwise linking numbers equal zero, and with framing $1/q_i$ on component $L_i$, for $q_i\in\Z$. Hoste \cite{Hoste:Casson} proved that the Casson invariant $\lambda$ of the integer homology sphere $Y_{1/q_1, \cdots, 1/q_n}(\L)$ satisfies a state sum formula,
\begin{equation}
\label{statesum}
	\lambda(Y_{1/q_1, \cdots, 1/q_n}(\L) ) = \lambda(Y) + \sum_{\L'\subset\L}\left( \prod_{i\in\L'}q_i \right) a_2(\L';Y), 
\end{equation}
where the sum is taken over all sublinks $\L'$ of $\L$. 

For example, let $\L=L_{1}\cup L_{2}\cup L_{3}$ be a three-component algebraically split link  in $S^{3}$ with framings $q_{i}=1$. Formula \eqref{statesum} simplifies to 
\begin{equation*}
\lambda(S^{3}_{1, 1, 1}(\L))=a_{2}(\L)+a_{2}(L_{12})+a_{2}(L_{13})+a_{2}(L_{23})+a_{2}(L_{1})+a_{2}(L_{2})+a_{2}(L_{3}),
\end{equation*}
where $L_{ij}=L_{i}\cup L_{j}$.  
Theorem \ref{Brunnianh}(a) immediately implies

\begin{corollary}
\label{casmu}
Assume that $\L$ is a Brunnian L-space link with three components. Then 
$$\lambda(S^{3}_{1, 1, 1}(\L))=\mu^{2}_{123}(\L)=\sum_{\bm{s}}h_{\L}(\bm{s}).$$
\end{corollary}

\begin{example}
Let $\L$ denote the Borromean rings, which are easily checked to form an L-space link.  Since $S^3_{1,1,1}(\L)$ is the Poincar\'e homology sphere, we see that $\lambda(S^3_{1,1,1}(\L)) = 1$.  This confirms the well-known calculation that $|\mu_{123}(\L)| = 1$.  From this, we also deduce that the $h$-function satisfies $h(0,0,0) = 1$ and vanishes elsewhere.
\end{example}

%%%%%%%%%%%%%%%%%%%%%%%%%%%%%%%%%
%%%%%%%%%%%%%%%%%%%%%%%%%%%%%%%%%
%%%%%%%%%%%%%%%%%%%%%%%%%%%%%%%%%
%%%%%%%%%%%%%%%%%%%%%%%%%%%%%%%%%

\section{Detection results}
\label{sec:detection}

In this section, we will apply the statements from the section above to show that for Brunnian links or algebraically split $L$-space links, sometimes information about surgery or Milnor invariants is sufficient for link detection. In Section~\ref{subsec:whitehead}, we focus on Whitehead link characterizations, in Section~\ref{subsec:borromean}, we focus on Borromean ring detections, and finally, in Section~\ref{subsec:morecomponents}, we discuss analogous results for links with more components.  Combining the results in these sections, we obtain proofs of Theorem \ref{Brunnian detection from Poincare surgery combined}, Theorem~\ref{Lspace link detection combined}, and Theorem \ref{whitehead or trefoil}.

Before we move to these statements, as a warm-up, we first determine which Brunnian links admit a rational surgery which is the three-sphere $S^{3}$. 

\begin{proposition}
\label{unlinksphere2}
Let   $\L=L_{1}\cup \cdots \cup L_{\ell}$ be an $\ell$-component  Brunnian link, and suppose that some surgery $S^{3}_{1/m_{1}, \cdots, 1/m_{\ell}}(\L)$ is the three sphere $S^{3}$, where all $m_i \neq 0$. Then $\L$ is the unlink.
\end{proposition}

\begin{proof}
We first prove that $S^{3}_{1, \cdots, 1}(\L)=S^{3}$. Let $L'_{\ell}$ be the image of $L_{\ell}$ in $S^{3}_{1/m_{1}, \cdots, 1/m_{\ell-1}}$ which is $S^{3}$ since $\L$ is Brunnian. Then $L'_{\ell}$ is the unknot by Gordon-Luecke \cite{GordonLuecke}. It is easy to see that 
$$S^{3}_{1/m_{1}, \cdots, 1/m_{\ell-1}, 1}(\L)=S^{3}_{1}(L'_{\ell})=S^{3}.$$
By repeating this argument, one can easily prove that $S^{3}_{1, \cdots, 1}(\L)=S^{3}$. By Example \ref{ex:Lspacelinks}, $\L$ is an L-space link.  Hence $\L$ is the unlink by Corollary \ref{unlinksphere}. \end{proof}

In the following sections, we will generalize to the case when $(1/m_1, \cdots, 1/m_\ell)$-framed Dehn surgery yields the Poincar\'e homology sphere. We will use the notation $\PHS=\Sigma(2,3,5)$ for the Poincar\'e homology sphere, oriented as the boundary of the positive-definite $E8$ plumbing, and $\Wh$ and $\B$ for the positive Whitehead and Borromean links, respectively. 

\subsection{Whitehead link detection}\label{subsec:whitehead}

This subsection is devoted to the detection of the Whitehead link. Proposition \ref{Whitehead by topology} informs and precedes the more general Proposition \ref{ifandonlyifwhite}, which gives the statement of Theorem \ref{Brunnian detection from Poincare surgery combined}(1) in the introduction. Corollary \ref{beta is 1} will give the statement of Theorem \ref{Lspace link detection combined}(1), and following this we prove Theorem \ref{whitehead or trefoil}.

\begin{proposition}
\label{Whitehead by topology}
Let $\L$ be a two-component Brunnian link.  If $S^3_{1,1}(\L)$ is the Poincar\'e homology sphere, then $\L$ is the positive Whitehead link $\Wh$.  
\end{proposition}
\begin{proof}
First, since $\L$ is Brunnian, each of the components $L_1$ and $L_2$ are unknotted.  Therefore, $S^3_n(L_i) = L(n,1)$ for all integers $n$.  For notation, let $E^{\L}$ denote the exterior of $\L$, and let $T_1$ and $T_2$ denote the boundary components corresponding to $L_1$ and $L_2$ respectively.  Let $L_{2,(n)}$ denote the image of $L_2$ in $S^3_n(L_1)$ and $M^{\L}_{(n)}$ the exterior of $L_{2,(n)}$ in $S^3_n(L_1)$.  Note that $M^{\L}_{(n)}$ has one torus boundary component, obtained by Dehn filling $E^{\L}$ along $T_1$.  We also will study $M^{\L}_{(\infty)}$, which is just a solid torus, since $L_{2,(\infty)}$ is unknotted. Since $S^3_{1,1}(\L)$ is the Poincar\'e homology sphere, $L_{2,(1)}$ is the trefoil by \cite{Ghiggini}.  Therefore, $M^{\L}_{(1)}$ is the exterior of the trefoil in $S^3$, and hence is a fibered three-manifold.   

We first claim that for all $n \neq \infty$, $L_{2,(n)}$ is a genus one fibered knot.  From the above discussion, we see that there are two slopes $\alpha, \beta$ on the boundary component $T_1$ in $E^{\L}$ such that Dehn filling along $\alpha$ results in a fibered three-manifold and filling along $\beta$ lowers the Thurston norm.  (Here, with respect to the canonical meridian-longitude coordinates on $T_1$, $\alpha = +1$ and $\beta = \infty$.)  By \cite[Theorem 1.4]{Ni}, we see that the core $J$ of the $\alpha$-filling sits on the fiber surface $F$ of the trefoil in $M^{\L}_{(1)}$ and that the framing $\beta$ corresponds to the surface framing of $F$.  It follows that any surgery on $J$ which is distance one from $\beta$ is also a genus one fibered three-manifold. Indeed, one cuts along $F$ and reglues by some number of Dehn twists along $J$.  
In meridian-longitude coordinates, $J$ is the core of $+1$-surgery on $L_1$ in the exterior of $L_2$, $\beta$ corresponds to the $\infty$-filling of $T_1$, and integral $n$-filling along $L_1$ for any $n$ is distance one from $\beta$. 
Therefore each $M^{\L}_{(n)}$ is a fibered three-manifold with fiber having genus one and a single boundary component.  Since $L_{2,(n)}$ is nullhomologous, we see that it is a genus one fibered knot in $L(n,1)$.  

By \cite[Theorem 4.3]{Baker}, for $n \neq 4, \infty$, we see that $L_{2,(n)}$ belongs to one of exactly two isotopy classes of knots. These must be either $\Wh_{2,(n)}$ or $\overline{\Wh}_{2,(n)}$ because the arguments above imply that integral surgery on a single component of $\Wh$ and $\overline{\Wh}$ give genus one fibered knots.  To see these are distinct, note that $+1$-surgery on $\Wh_{2,(n)}$ gives $S^3_n(T_{2,3})$ whereas $+1$-surgery on $\overline{\Wh}_{2,(n)}$ gives $S^3_n(4_1)$.  

In other words, $M^{\L}_{(n)} = M^{\Wh}_{(n)}$ or $M^{\overline{\Wh}}_{(n)}$ for infinitely many $n$.  By Lemma~\ref{lem:geometric-limit} below, $E^{\L} = E^{\Wh}$ or $E^{\overline{\Wh}}$.  In principle, this does not yet imply that $\L$ is a Whitehead link, since such links are not determined by their exteriors.  However, Lemma~\ref{lem:whitehead-exterior} shows that the additional condition that $S^3_{1,1}(\L)$ being the Poincar\'e homology sphere implies $\L$ is in fact $\Wh$.  
\end{proof}

\begin{lemma}\label{lem:geometric-limit}
If $M^{\L}_{(n)} = M^{\Wh}_{(n)}$ for infinitely many $n$, then $E^{\L} = E^{\Wh}$ and similarly for $\overline{\Wh}$.  
\end{lemma}
\begin{proof}
We do the case of $\Wh$.  The mirror is the same.

Suppose $E^{\L}$ is hyperbolic with two cusps. Then by Thurston's hyperbolic Dehn surgery theorem \cite{Thurston}, $M^{\L}_{(n)}$ is hyperbolic for all but finitely many $n$ with one cusp. In addition, $E^{\L}$ is the geometric limit of the sequence of complete hyperbolic manifolds $M^{\L}_{(n)}$ which is the same as $M^{\Wh}_{(n)}$. Hence, in the limit $E^{\L}=E^{\Wh}$.   

We next prove that $E^{\L}$ is hyperbolic. Suppose instead that $E^{\L}$ is not hyperbolic.  If $E^{\L}$ is reducible, then $\L$ is split, and so $\L$ is an unlink, and so $M^{\L}_{(n)} \neq M^{\Wh}_{(n)}$, contradiction.  Note that $E^{\L}$ is not Seifert, since otherwise $M^{\L}_{(n)}$ would not be hyperbolic for any $n$.  It follows that there is a non-boundary parallel incompressible torus (i.e. essential torus) in $E^{\L}$.  Consider the piece $X$ of the JSJ decomposition of $E^{\L}$ which contains $T_1$.  Let $T'$ denote a boundary component of $X$ which is not $T_1$ or $T_2$ and essential in $E^{\L}$.
 Then, $T'$ remains incompressible and non-boundary parallel in all but finitely many fillings of $X$ along $T_1$, unless $X$ is a cable space (i.e. a Seifert fibered space with base orbifold an annulus and exactly one cone point)  \cite[Theorem 2.4.4]{CGLS}.  If $X$ is not a cable space, then generically $M^{\L}_{(n)}$ has a non-boundary parallel incompressible torus. However,  $M^{\Wh}_{(n)}$ is hyperbolic for infinitely many $n$, and hence does not have a non-boundary parallel incompressible torus, contradiction.  If instead $X$ is a cable space, and infinitely many of the $n$-fillings of $T_1$ are not those that cause $T'$ to compress in $M^{\L}_{(n)}$, then again $M^{\L}_{(n)}$ has a non-boundary parallel incompressible torus for infinitely many $n$, while $M^{\Wh}_{(n)}$ does not, contradiction.  

Lastly, suppose that $X$ is a cable space, and that infinitely many integral $n$-slopes on $L_1$ are slopes on $X$ which cause $T'$ to compress.  The fiber slope on a boundary component of a cable space is uniquely characterized by the slope which is distance one from at least 3 different compressing slopes \cite[Theorem 2.4.3]{CGLS}.  This means that the $\infty$-slope on $L_1$ corresponds to the fiber slope $\phi$ on $X$.  
Recall that for a cable space, filling one boundary component along the fiber slope results in the connected sum of a solid torus and a non-trivial lens space. Let $X(\phi)$ denotes the Dehn filling of $T_{1}$ with the fiber slope $\phi$. Therefore, $M^{\L}_{(\infty)}$ contains $X(\phi)$ as a codimension zero submanifold  which has a lens space summand, but $M^{\L}_{(\infty)}$ is a solid torus, so this is a contradiction.
\end{proof}

\begin{lemma}\label{lem:whitehead-exterior}
If $\L$ is a two-component link in $S^3$ with $E^{\L} = E^{\Wh}$ or $E^{\overline{\Wh}}$ and $S^3_{+1,+1}(\L)$ is the Poincar\'e homology sphere, then $\L = \Wh$.  
\end{lemma}
To explain why this lemma is necessary, notice that doing $1/n$-surgery on a single component of $\Wh$ yields $S^3$, and that the image of the other component is a twist knot $K[2n,-2]$. The core of that surgery, together with this twist knot, is a two-component link in $S^3$ with the same exterior as $\Wh$.  
\begin{proof}
We begin with the case that $E^{\L} = E^{\Wh}$.  Since $Wh$ has linking number 0, any $\L$ in $S^3$ with the same exterior can be described by the core of $(1/m,1/n)$-surgery on $\Wh$, where that surgery results in $S^3$.  Note that $S^3_{\frac{1}{m},\frac{1}{n}}(\Wh)$ is $1/m$-surgery on the twist knot $K[2n,-2]$ (see Figure \ref{fig:twistknot}), and thus either $1/m = \infty$ or $1/n = \infty$.  Without loss of generality, $1/n = \infty$.  Therefore, 
\[
\PHS= S^3_{1,1}(\L) = S^3_{\frac{1}{m+1}, 1}(Wh) = S^3_{\frac{1}{m+1}}(K[2,-2]) = S^3_1(K[2(m+1), -2]).
\]  
It follows that $K[2(m+1), -2]$ is the right-handed trefoil, and hence $m = 0$.  

We can repeat a similar argument with $\overline{\Wh}$.  In this case, we see that 
\[
\PHS = S^3_{1,1}(\L) = S^3_{\frac{1}{m+1}, 1}(\overline{\Wh}) = S^3_{\frac{1}{m+1}}(4_1) = S^3_1(4_1),
\]               
which is a contradiction.
\end{proof}

\begin{figure}
\centering
	\begin{tikzpicture}
        \node[anchor=south west,inner sep=0] at (0,0) {\includegraphics[width = 4cm]{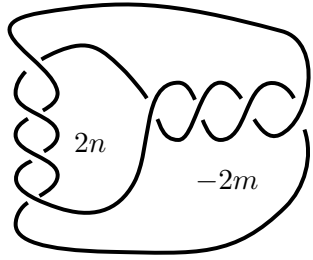}};
        \draw(1.1,1.5) node {$2n$};
        \draw(2.9,1) node {$-2m$};  									
	\end{tikzpicture}
\caption{The generalized twist knot $K[2n, -2m]$ is a two-bridge knot described by numerator closure of the rational tangle with continued fraction $[2n, -2m]$). For example, $K[2,-2]$ is the right-handed trefoil.}
\label{fig:twistknot}
\end{figure}

By an argument similar to the one in Proposition \ref{Whitehead by topology}, we now can characterize which rational surgeries on a Brunnian two-component link are the Poincar\'e homology sphere. 

\begin{proposition}
\label{ifandonlyifwhite}
Suppose that $\L$ is a two-component Brunnian link. Then $S^{3}_{1/m, 1/n}(\L)$ is the Poincar\'e homology sphere if and only if  $mn=1$  and $\L$ is the Whitehead link or its mirror. 
\end{proposition}

\begin{proof}
The if part is easy. For the only if part, we first claim that if $S^{3}_{1/m, 1/n}(\L)$ is the Poincar\'e homology sphere, then $m=\pm 1$ and $n=\pm 1$. Suppose $\L=L_{1}\cup L_{2}$, and let $L'_{1}$ denote the image of $L_{1}$ in $S^{3}$ after blowing down $L_{2}$. Then $S^{3}_{1/m, 1/n}(\L)=S^{3}_{1/m}(L'_{1})= \pm \PHS$, which implies that $m=\pm 1$, since only a trefoil has a surgery to the Poincar\'e homology sphere, and the surgery coefficient is integral \cite{Ghiggini}.  A similar argument applies for $n$. If $m=n=1$, by Proposition \ref{Whitehead by topology}, $\L$ must be the Whitehead link. If $m=n=-1$, we have $S^{3}_{-1, -1}(\L)=S^{3}_{1, 1}(\bar{\L})$ where $\bar{\L}$ is the mirror of $\L$. Then $\bar{\L}$ must be the Whitehead link, and hence, $\L$ is the mirror of the Whitehead link. 

Without loss of generality, we assume that $m=1, n=-1$. By a similar argument as the one in Proposition \ref{Whitehead by topology}, the exterior of the link $E^{\L}=E^{Wh}$ or $E^{\overline{Wh}}$. However, this is impossible by applying the same argument as in Lemma \ref{lem:whitehead-exterior}. 
\end{proof}

We now transition to characterizing the Whitehead link among two-component algebraically split L-space links.  

\begin{lemma}
\label{lem:whitehead detection}
Let $\L=L_{1}\cup L_{2}$ be an algebraically split L-space link with the following $h$-function:
\begin{equation}
\label{eq: h whitehead}
h_{\L}(s_1,s_2)=\begin{cases}
1 & \text{if}\ s_1=s_2=0\\
0 & \text{otherwise}.
\end{cases}
\end{equation}
Then $\L$ is the Whitehead link.
\end{lemma}

\begin{proof}
Let us prove that $S^{3}_{1,1}(\L)$ is the Poincar\'e homology sphere. Let $L'_{2}$ be the image of $L_2$ in $S^3_{1}(L_1)$, so that
$S^3_{1,1}(\L)=S^3_{1}(L'_2)$.

Since the $h$-function of $\L$ is given by \eqref{eq: h whitehead}, the link surgery complex for $S^{3}_{1,1}(\L)$  can be truncated so that it contains only $\A^{-}_{00}(\L)$ \cite{GLM} . Therefore $S^{3}_{1}(L'_{2})=S^3_{1,1}(\L)$ is an L-space with $d$-invariant equal to $-2H(0,0) = -2$. 
The knot $L'_{2}$ is an L-space knot and the L-space surgery coefficient of $+1$ satisfies $1\geq 2g(L'_{2})-1$. Because the surgery $S^{3}_{1}(L'_{2})$ is not $S^3$, we must have the genus of $L'_{2}$ is exactly 1, and so $L'_{2}$ is the right-handed trefoil. 
This also shows $S^3_{1,1}(\L)$ is the Poincar\'e homology sphere.  Finally, since the unknot is determined by its $h$-function among L-space knots, Lemma~\ref{h-function list}\eqref{h-function bdy} implies that the components of $\L$ are unknotted.  
Now the statement follows from Proposition \ref{Whitehead by topology}.
\end{proof}

\begin{corollary}
\label{beta is 1}
Suppose that $\L$ is a two-component L--space link with unknotted components and $\beta(\L)=\pm 1$. 
Then $\L$ is the Whitehead link. 
\end{corollary}

\begin{proof}
Since $\L$ is an L--space link with unknotted components, we have $\sum_{\bm{s}} h(\bm{s})=\pm \beta(\L)=\pm 1$.
Since $h(\bm{s})\ge 0$ and $h(-\bm{s})=h(\bm{s})$, the only possible $h$-function is given by \eqref{eq: h whitehead}. Hence, $\L$ is the Whitehead link by Lemma \ref{lem:whitehead detection}. 
\end{proof}

\begin{reptheorem}{whitehead or trefoil}
Let $\L=L_{1}\cup L_{2}$ be an algebraically split L-space link such that $S^{3}_{1, 1}(\L)$ is the Poincar\'e homology sphere. Then $\L$ is either the Whitehead link or the split union of $T(2, 3)$ and the unknot.
\end{reptheorem}

\begin{proof}
By \cite[Proposition 5.6]{GLM} we have that either $S^3_1(L_1)$ or $S^3_1(L_2)$ is an L--space. Without loss of generality, in the remainder of the proof we assume that $S^3_1(L_1)$ is an L--space. Then $L_1$ is an L--space knot of genus 0 or 1, so it is either unknotted or the right-handed trefoil.

{\bf Case 1:} $L_1$ is an unknot. In this case we can blow it down and obtain a knot $L'_2$ such that $S^3_{1,1}(\L)=S^3_1(L'_2)$. This means that $S^3_1(L'_2)$ is the Poincar\'e homology sphere.  So $L'_{2}$ is $T(2,3)$.
 
By \cite[Theorem 4.8]{GLM} the $H$-function for $L'_2$ equals $H_{\L}(0,s_2)$, so 
\[
h_{\L}(0,s_2)=h_{T(2,3)}(s_2)=\begin{cases}
1 & \text{if}\ s_2=0\\
0 & \text{otherwise}.\\
\end{cases}
\]
By Lemma \ref{h-function  list}\eqref{lem: h increases} we get $h_{\L}(N,s_2)=0$ for all $N\ge 0$ and $s_2\neq 0$, and hence
$h_2(s_2)=0$ for $s_2\neq 0$. Since $L_2$ is an L--space knot, this implies the genus is at most one, and hence it is either unknotted or $T(2,3)$.

\begin{figure}[ht]
\begin{tikzpicture}
\draw (0,0) node {0};
\draw (0,2.5) node {0};
\draw (0,3.5) node {0};
\draw (0,3) node {1};
\draw (0,6) node {0};
\draw (0,1.5) node {$\vdots$};
\draw (0,4.5) node {$\vdots$};
\draw (1.5,0) node {$\cdots$};
\draw (1.5,3) node {$\cdots$};
\draw (1.5,6) node {$\cdots$};
\draw (3,0) node {0};
\draw (3,2.5) node {0};
\draw (3,3) node {1};
\draw (3,1.5) node {$\vdots$};
\draw (3,4.5) node {$\vdots$};
\draw (3,3.5) node {0};
\draw (3,6) node {0};
\draw (4.5,0) node {$\cdots$};
\draw (4.5,3) node {$\cdots$};
\draw (4.5,6) node {$\cdots$};
\draw (6,0) node {0};
\draw (6,3) node {1};
\draw (6,6) node {0};
\draw (6,1.5) node {$\vdots$};
\draw (6,4.5) node {$\vdots$};
\draw (6,2.5) node {0};
\draw (6,3.5) node {0};
\draw [dotted,->] (-0.5,3)--(6.5,3);
\draw [dotted,->] (3,-0.5)--(3,6.5);
\end{tikzpicture}
\caption{The $h$-function for $O\sqcup T(2,3)$}
\label{fig: h unknot trefoil}
\end{figure}
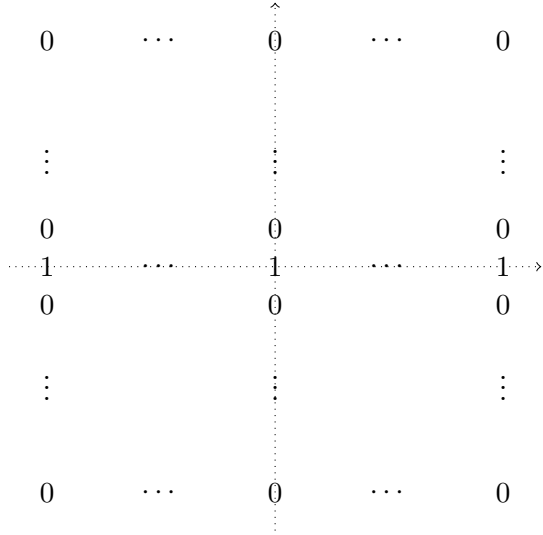

If $L_2$ is unknotted, then $\L$ is the Whitehead link by Proposition \ref{Whitehead by topology}.

If $L_2$ is $T(2,3)$, then $h_2(0)=h_{\L}(N,0)=1$ for all $N>0$, and Lemma \ref{h-function list}\eqref{h-symmetry}
implies that $h_{\L}(s_1,s_2)$ coincides with the $h$-function of the disjoint union of $O\sqcup T(2,3)$ (see Figure \ref{fig: h unknot trefoil}).
By the main result of \cite{Beibei3} the Thurston polytope of $\L$ is the same as the Thurston polytope of 
$O\sqcup T(2,3)$. In particular, $L_1$ bounds a disk not intersecting $L_2$, and hence $\L=O\sqcup T(2,3)$.

{\bf Case 2:} $L_1$ is the trefoil. Let $L'_2$ be the knot corresponding to $L_2$ in $S^3_1(L_1)=\PHS$.
Since $+1$-framed surgery along $L_2'\subset \PHS$ yields $\PHS$, by the Dehn surgery characterization of the unknot in an integer homology sphere L-space \cite{KMOS, Gainullin}, $L_2'$ is an unknot. 

Now for all $d$ we have $S^3_{1,d}(\L)=\PHS_d(L'_2)=S^3_{1,d}(T(2,3)\sqcup O).$ Similarly to \cite[Theorem 4.8]{GLM} we conclude that 
\[
h_{\L}(0,s_2)=h_{T(2,3)\sqcup O}(0,s_2)=1
\]
as in Figure \ref{fig: h hypothetical}. 

Now we claim that $L_{2}$ is the unknot. Otherwise, suppose that $g(L_{2})\geq 1$. 
By Lemma \ref{h-function  list}\eqref{lem: h increases}, 
$h_2(s_{2})=h_{\L}(N, s_{2})\leq h_{\L}(0, s_{2})=1.$
More precisely, $h_2(s_{2})=0$ for all $|s_{2}|\geq g(L_{2})$ and $h_2(s_{2})=1$ otherwise as in Figure \ref{fig: h hypothetical}. 

Recall that $h'(s_{1}, s_{2})=h_{\L}(s_{1}, s_{2})-h_{1}(s_{1})-h_{2}(s_{2})$. Since $L_{1}$ is the trefoil, $h_{1}(s_{1})=0$ for all $s_{1}\neq 0$.  By a similar argument to the  one in Lemma \ref{h-function  list}\eqref{lem: h increases}, one can prove that $h'(\bm{s}-\bm{e}_{1})\geq  h'(\bm{s})$  for all $s_{1}>1$. Note that $h'(\infty, s_{2})=0$ for all $s_{2}\in \Z$. Then $h'(s_{1}, s_{2})\geq 0$ for all $s_{1}\geq1$. By Lemma \ref{h-function list}\eqref{h-symmetry}, $h'(s_{1}, s_{2})\geq 0$ for all $|s_{1}|\geq 1$.  

By \eqref{statesum}, 
$$1=\lambda(S^{3}_{1, 1}(\L))=a_{2}(\L)+a_{2}(L_{1})+a_{2}(L_{2}). $$

Note that $a_{2}(L_{1})=\sum h_1(s_{1})=1$, $a_{2}(L_{2})=\sum h_2(s_{2})=2g(L_{2})-1$. Then $a_{2}(\L)=\sum h'_{\L}(\bm{s})=1-2g(L_{2})<0$. Observe that (see Figure \ref{fig: h hypothetical})
$$\sum_{s_{2}\in \Z} h'_{\L}(0, s_{2})=1-2g(L_{2})=\sum_{\bm{s}\in \H(\L)} h'_{\L}(s_{1}, s_{2}),$$
and $h'_\L(s_{1}, s_{2})\geq 0$ for all $|s_{1}|\geq 1$.
Hence $h'_{\L}(s_{1}, s_{2})=0$ for all $s_{1}\neq 0$, indicating that $h_{\L}(s_{1}, s_{2})=h_{1}(s_{1})+h_{2}(s_{2})$ for all $s_{1}\neq 0$.  Note that $h_{1}(s_{1})=0$ for all $s_{1}\neq 0$. This implies that $h_{\L}(s_{1}, s_{2})=h_{2}(s_{2})$ for all $s_{1}\geq 1$, and the $h$-function has the form as in Figure \ref{fig: h hypothetical}. 

\begin{figure}[ht]
\begin{tikzpicture}
\draw (0,0) node {0};
\draw (0,2.5) node {1};
\draw (0,3.5) node {1};
\draw (0,2) node {0};
\draw (0,4) node {0};
\draw (0,3) node {\vdots};
\draw (0,6) node {0};
\draw (0,1.5) node {$\vdots$};
\draw (0,4.5) node {$\vdots$};
\draw (1.5,0) node {$\cdots$};
\draw (1.5,3) node {$\cdots$};
\draw (1.5,6) node {$\cdots$};
\draw (3,0) node {1};
\draw (3,2) node {1};
\draw (3,2.5) node {1};
\draw (3,3) node {\vdots};
\draw (3,1.5) node {$\vdots$};
\draw (3,4.5) node {$\vdots$};
\draw (3,3.5) node {1};
\draw (3,4) node {1};
\draw (3,6) node {1};
\draw (4.5,0) node {$\cdots$};
\draw (4.5,3) node {$\cdots$};
\draw (4.5,6) node {$\cdots$};
\draw (6,0) node {0};
\draw (6,3) node {$\vdots$};
\draw (6,6) node {0};
\draw (6,1.5) node {$\vdots$};
\draw (6,4.5) node {$\vdots$};
\draw (6,2.5) node {1};
\draw (6,3.5) node {1};
\draw (6,2) node {0};
\draw (6,4) node {0};
\draw [dotted,->] (-0.5,3)--(6.5,3);
\draw [dotted,->] (3,-0.5)--(3,6.5);
\draw [dotted,->] (7,3.5)--(6.2,3.5);
\draw [dotted,->] (7,2.5)--(6.2,2.5);
\draw (7.5,3.5) node {\scriptsize{$g(L_2)-1$}};
\draw (7.5,2.5) node {\scriptsize{$1-g(L_2)$}};
\draw (2.5,0) node {0};
\draw (2.5,2) node {0};
\draw (2.5,2.5) node {1};
\draw (2.5,3) node {$\vdots$};
\draw (2.5,3.5) node {1};
\draw (2.5,4) node {0};
\draw (2.5,6) node {0};
\draw (3.5,0) node {0};
\draw (3.5,2) node {0};
\draw (3.5,2.5) node {1};
\draw (3.5,3) node {$\vdots$};
\draw (3.5,3.5) node {1};
\draw (3.5,4) node {0};
\draw (3.5,6) node {0};
\end{tikzpicture}
\caption{Hypothetical $h$-function for Case 2}
\label{fig: h hypothetical}
\end{figure}

As in \cite[Section 5]{GLM} we may define 
\[
	b_{1}=\min \{s_{1}-1 \mid H_{\L}(s_{1}, s_{2})=H_2(s_{2}) \text{ for all } s_2\}.
\]
Clearly, from Figure \ref{fig: h hypothetical} we have $b_1=0$. It is proved in \cite[Proposition 4.7]{Liu19} that (under some assumptions on the $h$-function which are satisfied in this case) if 
$S^3_{d_1,d_2}(\L)$ is an L--space for $d_1>2b_1$ and $d_2\ll 0$ then $L_2$ is the unknot.
Since $(1, d)$-surgery on $\L$ yields $\PHS\# L(d, 1)$, which is an L-space for any nonzero integer $d$,  $L_{2}$ is the unknot. Hence, by the same argument as the one in Case 1, $\L$ is the disjoint union of the unknot and $T(2, 3)$. 
\end{proof}

\subsection{Borromean link detection}\label{subsec:borromean}
In this subsection, Proposition \ref{Borromean by topology} will inform the more general Proposition \ref{thm:rational} which corresponds with Theorem \ref{Brunnian detection from Poincare surgery combined}(2) in the introduction. Proposition \ref{borromean} will give the statement of Theorem \ref{Lspace link detection combined}(2).

\begin{proposition}
\label{Borromean by topology}
Let $\L$ be a three-component Brunnian link.  Then if $S^3_{1,1,1}(\L)$ is the Poincar\'e homology sphere, then $\L$ is the Borromean rings $\B$.  
\end{proposition}

\begin{proof}
Because $L_1\cup L_2\cup L_3$ is Brunnian, all proper two-component sublinks $L_i\cup L_j$ are unlinks and $S^3_{m,n}(L_i\cup L_j) = L(m,1) \# L(n, 1)$ for all integers $m, n$. We will use notation similar to that of Proposition \ref{Whitehead by topology}. Let $E^{\L}$ denote the exterior of $\L$, and $T_i$ the boundary component corresponding to $L_i$. We write $L_{3, (m), (n)}$ to denote the image of $L_3$ in $S^3_{m,n}(L_1\cup L_2) = L(m,1) \# L(n, 1)$, and write $M^{\L}_{(n), (m)}$ for the exterior of $L_{3, (m), (n)}$ in $L(m,1) \# L(n, 1)$. We also write $L_{3,2,(n)}$ to denote the image of $L_2\cup L_3$ in $S^3_{n}(L_1)$. We similarly have that $M^{\L}_{(\infty),(\infty)}$ is a solid torus. Because $S^3_{1,1,1}(\L)$ is the Poincar\'e homology sphere, $L_{3,(1),(1)}$ is the trefoil knot and $M^{\L}_{(1),(1)}$ is the exterior of the trefoil in $S^3$, a fibered three-manifold.   

We will invoke \cite[Theorem 1.4]{Ni} to argue that for all pairs of integers $n,m \neq \infty$, $L_{3,(n),(m)}$ is a genus one fibered knot. The argument is analogous to that of Proposition \ref{Whitehead by topology}. 
There are two slopes $\alpha_1, \beta_{1}$ on the boundary component $T_{1}$ in $E^{\L}$ such that Dehn filling along $\alpha_{1}$ results in a fibered three-manifold and filling along $\beta_{1}$ lowers the Thurston norm. With respect to the canonical meridian-longitude coordinates on $T_{1}$, $\alpha_{1}=+1$ and $\beta_{1}=\infty$. Indeed, filling along $\alpha_{1}$ results in the complement of Whitehead link in $S^{3}$.  Since $\L$ is Brunnian, then $L_{3, 2, (1)}$ is also Brunnian and $S^{3}_{1, 1}(L_{3, 2, (1)})=S^{3}_{1, 1, 1}(\L)$ is the Poincar\'e homology sphere. By Proposition  \ref{Whitehead by topology}, $L_{3, 2, (1)}$ is the Whitehead link. Filling along $\beta_1$ results in the complement of two-component unlink in $S^{3}$, which  lowers the Thurston norm. By \cite[Theorem 1.4]{Ni}, the core $J_1$ of the $\alpha_1$-filling of $T_1$ sits on the fiber surface $F$ of the Whitehead link complement  $M^{\L}_{(1)}$. The framing $\beta_1$ corresponds to the surface framing of $F$.  Now we repeat this argument. In particular, $M^{\L}_{(1)}$ has two boundary components $T_2$ and $T_3$. We consider two fillings $\alpha_2$ and $\beta_2$ along $T_2$ similary (i.e. $\alpha_{2}=+1, \beta_{2}=\infty$). By the argument given above, the $\alpha_2$ filling $M^{\L}_{(1),(1)}$ yields a fibered three-manifold which is the exterior of the right-handed trefoil and is obtained by capping off one boundary of the fiber surface of the Whitehead link.  The $\beta_2$ filling $M^{\L}_{(\infty),(1)}$ which results in the complement of the unknot in $S^{3}$ reduces the Thurston norm. By a second application of \cite[Theorem 1.4]{Ni}, we have that the core $J_2$ of the $\alpha_2$-filling of $T_2$ sits on the fiber surface $F_3$ in $M^{\L}_{(1), (1)}$ and any surgery which is distance one from $\beta_2$ will produce a fibered three-manifold. In fact, any fillings along $T_1, T_2$ which are distance one from $\beta_1, \beta_2$ will extend the monodromy of the fibration. We now have for any integers $m, n$ that $M^{\L}_{(m),(n)}$ is fibered and that the fiber is genus one. So $L_{3,(m),(n)}$ is a genus one fibered knot in $L(m,1) \# L(n,1)$.

Having established that $L_{3,(m),(n)}$ is genus one fibered knot, we turn again to Baker's classification of genus one fibered knots in lens spaces \cite{Baker}. By \cite[Lemma 2.2]{Baker}, the isotopy classes $(M, K)$ of genus one fibered knots $K$ in $M$ are in a one-to-one correspondence with isotopy classes $(N, A)$ where $N$ is a link in $S^3$ whose branched double cover is $M$, and with a braid axis $A$ that yields a closed 3-braid presentation of $N$. Fix $m, n$ odd and distinct.  Applying this here, we must have that $N$ is the connected sum of torus knots $T(2, m) \# T(2, n)$. Indeed, $T(2, m)$ is the unique knot whose branched double cover is the lens space $L(m, 1)$ by Hodgson and Rubinstein \cite{HR85}. By the equivariant sphere theorem for involutions \cite{KT}, $T(2, m) \# T(2, n)$ is the unique knot with branched double cover $L(m, 1) \# L(n, 1)$. Thus, we seek to identify the braid axes $A$ yielding closed 3-braid presentations for $T(2, m) \# T(2, n)$. There is a unique braid axis  for closed 3-braid presentations of  the knot $T(2, m)\# T(2, n)$ with odd integers satisfying $|m|, |n| > 1$. Note that the torus knots $T(2, m)$ and $T(2, n)$ are strongly invertible, so $T(2, m)\# T(2, n)$ is invertible. There is no distinction between the orientations of the knot. By \cite[Lemma 3.6]{Baker}, there is at most one equivalence class of braid axes giving 3-braid representatives for the oriented knot and its inverse. 

Similarly to the argument used for the Whitehead link, we have that $M^{\L}_{(m), (n)}=M^{\B}_{(m), (n)}$ for infinitely many $m, n$. An argument completely analogous to that of Lemma \ref{lem:geometric-limit} implies that the geometric limit $E^{\L}$ of the hyperbolic manifolds $M^{\L}_{(m), (n)}$ agrees with the limit $E^{\B}$ of the sequence $M^{\B}_{(m), (n)}$, meaning the exterior $E^{\L}$ is the same as the Borromean exterior $E^{\B}$. The result now follows from Lemma \ref{lem:borromean-exterior} below.
\end{proof}

\begin{lemma}\label{lem:borromean-exterior}
If $\L$ is a three-component Brunnian link in $S^3$ with $E^{\L} = E^{\B}$ and $S^3_{1,1,1}(\L)$ is the Poincar\'e homology sphere, then $\L = \B$.  
\end{lemma}

\begin{proof}
Any link $\L$ in $S^3$ with the same exterior as $\B$ can be described by the core of $(1/m,1/n,1/p)$-surgery on $\B$, where this surgery results in $S^3$.  By Proposition \ref{unlinksphere2}, at least one of $m, n, p$ is zero. Without loss of generality, we assume $p=0$. Then
\[
\PHS = S^3_{1,1,1}(\L) = S^3_{\frac{1}{m+1}, \frac{1}{n+1}, 1}(\B) = S^3_{\frac{1}{m+1}, \frac{1}{n+1}}(\Wh).
\]  
By Proposition \ref{ifandonlyifwhite}, $m=n=0$.   Therefore, $\L=\B$. 
\end{proof}

Based on the argument for the detection of the Borromean rings, we now characterize which rational surgeries on a Brunnian three-component link are the Poincar\'e homology sphere. 

\begin{proposition}
\label{thm:rational}
Suppose that $\L$ is a three-component Brunnian link. Then $S^{3}_{1/p, 1/q, 1/r}(\L)$ is the Poincar\'e homology sphere if and  only if  the link $\L$ is the Borromean rings, and $|p|=|q|=|r|=1$ with the same sign. 
\end{proposition}

\begin{proof}
The if part is easy. For the only if part, we first claim that $p, q, r=\pm 1$. Suppose that $\L=L_1\cup L_2\cup L_3$, and let $L'_{1}$ denote the image of $L_{1}$ in $S^{3}$ after blowing down $L_{2}, L_{3}$. Then $S^{3}_{1/p, 1/q, 1/r}(\L)=S^{3}_{1/p}(L'_{1})$, which implies that $1/p=\pm 1$.  Hence, $p=\pm 1$. Further, the sign of $p$ is determined by the definiteness of the plumbing which the surgered Poincar\'e homology sphere bounds.  A similar argument can be used to prove that $q, r=\pm 1$, and that their signs must agree with that of $p$. If $p=q=r=1$, by Proposition \ref{Borromean by topology}, $\L$ is the Borromean rings. If $p=q=r=-1$, then the mirror of $\L$ is the Borromean rings, which implies that $\L$ is also the Borromean rings.  
\end{proof}

\begin{proposition}
\label{borromean}
Assume that $\L$ is a Brunnian L--space link with three components and $\mu_{123}(\L)=\pm 1$. Then $\L$ is the Borromean rings.
\end{proposition}

\begin{proof}
By Theorem \ref{Brunnianh}(a),  $\sum_{\bm{s}}h(\bm{s})=1$. Recall that $h(\bm{s})\geq 0$, and takes the maximal value at $h(0, 0, 0)$. So
\[
h(\bm{s})=\begin{cases}
1\ \text{if}\ \bm{s}=(0,0,0)\\
0\ \text{otherwise}.\\
\end{cases}
\]

Let us prove that $S^3_{1, 1, 1}(\L)$ is the Poincar\'e homology sphere. 
Because $\L$ is Brunnian and the $h$-function agrees with that of the Borromean rings, the link surgery complex can be truncated so that it contains only $\A^{-}_{000}(\L)$. 
Just as with the proof of Lemma \ref{lem:whitehead detection}, because $S^3_{1,1,1}(\L)$ is an L-space with $d$-invariant $-2H(0,0,0) = -2$ we again have that $S^3_{1,1,1}(\L)$ is the Poincar\'e sphere. That is, by doing $+1$-surgery along the link components $L_1$ and $L_2$, we have a knot $L_3''\subset S^3$ with a $+1$-surgery to an L-space with $d=-2$. Thus $L_3''$ is the trefoil and  $S^3_{1,1,1}(\L)$ is the Poincar\'e sphere.
By Proposition \ref{Borromean by topology}, $\L$ is isotopic to the Borromean rings.
\end{proof}

\subsection{Links with four or more components}\label{subsec:morecomponents}

In this section we show that the above results have no analogues for links with more than three components. Proposition \ref{thm: 4 comp brunnian} gives Theorem \ref{Lspace link detection combined}(3) and Proposition \ref{thm: 4 comp no phs} gives Theorem \ref{Brunnian detection from Poincare surgery combined}(3) in the introduction.
\begin{proposition}
\label{thm: 4 comp brunnian}
Assume that $\L$ is a Brunnian L--space link with four or more components. Then $\L$ is the unlink.
\end{proposition}

\begin{proof}
By Example \ref{ex: 4comp} we have $a_2(\L)=\widetilde{\Delta}'(1,\ldots,1)=0$. Therefore by Corollary \ref{cor: Delta' Brunnian} we have $\sum_{\bm{s}}h(\bm{s})=0$ and $h(\bm{s})$ vanishes for all $\bm{s}$. Then by \cite[Theorem 1.3]{Beibei} $\L$ is the unlink.
\end{proof}

Now we consider whether rational surgery on a Brunnian link with at least four components is the Poincar\'e homology sphere. Note that  if  $\L$ is a Brunnian link and $S^3_{1,\ldots,1}(\L)$ is an L--space, then by Theorem \ref{l-space link cond} it is an L--space link.  By Proposition \ref{thm: 4 comp brunnian}, $\L$ is the unlink, and $(1, \cdots, 1)$-surgery cannot be the Poincar\'e homology sphere. For general rational surgeries, we have the similar result. 

\begin{proposition}
\label{thm: 4 comp no phs}
Let $\L$ be an $n$-component  Brunnian link with $n\geq 4$.  Then $S^3_{1/m_{1},\ldots,1/m_{n}}(\L)$ is not the Poincar\'e homology sphere.
\end{proposition}

\begin{proof}

Suppose $S^3_{1/m_{1},\ldots,1/m_{n}}(\L)$ is the Poincar\'e homology sphere.  By the same argument as in Proposition \ref{thm:rational}, $m_{i}=\pm 1$ where $i=1, \cdots, n$. 
Since $S^3_{1/m_{1},\ldots, 1/m_{n}}(\L) = S^3_{1/m_{2},\ldots, 1/m_{n}}(L'_2 \cup \ldots \cup L'_n)$ and the latter is still a Brunnian link, it suffices to consider the case that $\L$ has four components.  
By repeating the arguments and notation from Proposition~\ref{Whitehead by topology} and Proposition~\ref{Borromean by topology}, we see that $L_{4,(l),(p),(q)}$ is a genus 1 fibered knot in $(l ,p ,q)$-surgery on the three-component unlink whenever $l, p, q$ are integers.  However, $L(l, 1)\# L(p, 1)\# L(q, 1)$ does not contain a genus 1 fibered knot since its Heegaard genus is 3 when $|l|, |p|, |q| > 1$, which is a contradiction. Hence, $S^3_{1/m_{1},\ldots,1/m_{n}}(\L)$ is not the Poincar\'e  homology sphere. 
\end{proof}

%%%%%%%%%%%%%%%%%%%%%%%%%%%%%%%%%
%%%%%%%%%%%%%%%%%%%%%%%%%%%%%%%%%
%%%%%%%%%%%%%%%%%%%%%%%%%%%%%%%%%
%%%%%%%%%%%%%%%%%%%%%%%%%%%%%%%%%

\section{From triple linking to $d$-invariants}
\label{sec:d}

In this section, unlike the rest of the paper, we will use $d$-invariants for Heegaard Floer homology with a variety of coefficients.  See Section \ref{subsec:coefs} for more details on subtleties of coefficient fields. Some of the results hold over an arbitrary field $\kk$, and we will denote the corresponding $d$-invariants by $d_{\kk}$.

Let us also introduce some additional notation required in this section. Given a nullhomologous link $\L = L_1\cup \cdots \cup L_\ell$ in a rational homology sphere $Y$ and integers $\bm{m}=(m_1, \cdots, m_\ell)$, let $W_{\bm{m}}$ denote the four-dimensional two-handle cobordism from $Y$ to $Y_{\bm{m}}(\L)$ which is the surgery manifold specified by the $\bm{m}$-framed link $\L$. The notation $\mathfrak{s}, \mathfrak{t}$ and $\mathfrak{w}, \mathfrak{z}$ will generally be used to denote $\spinc$-structures on three-manifolds and four-manifolds, respectively. In particular, let $\mfw$ denote a $\spinc$-structure on $W_m=W_m(K)$ which extends $\mfs$ in $\spinc(Y)$, where $K$ is a knot in $Y$. Recall from \cite[Theorem 4.2]{OS:Integer} that $\mfw$ induces a $\spinc$-structure $\mfs_i$ on $Y_m(K)$ indexed by
\[
	\langle c_1(\mfw), [F] \rangle + m =2i,
\]
where $[F]$ is the surface in $W_m(K)$ coming from capping off a Seifert surface for $K$.  We can similarly define $\mfs_{\bm{i}}=\mfs_{i_1,\cdots, i_\ell}$ as the restriction of $\spinc$-structures induced by surgery along links of $\ell$ components by a similar formula.

We will denote by $\delta(W, \mfw)$ the quantity
\[
	\delta(W, \mfw) = \frac{c_1(\mfw)^2 - 3\sigma(W)-2\chi(W)}{4},
\]
which describes the shift in absolute grading induced by the $\spinc$-cobordism map on the Floer homology associated to $(W,\mfw)$.

\subsection{$d$-invariants for standard three-manifolds}
\label{subsec:d std}

In this section we review the definition of $d$-invariants for \emph{standard} three-manifolds from \cite[Section 9]{OS03}. An additional reference for this material can be found in \cite{Levine}.

Let $H$ be a finitely generated, free abelian group and let $\Lambda^{\ast}(H)$ denote the exterior algebra of $H$. If $Y$ is a three-manifold, we denote
\[
\Lambda^{\ast}H^1:=\Lambda^{\ast}H^1(Y;\Z),\ \Lambda^{\ast}H_1:=\Lambda^{\ast}\left(H_1(Y;\Z)/\textup{Tors}\right).
\] 
The module $HF^\infty_\kk(Y)$ is called \emph{standard} if for each torsion $\spinc$ structure $\mft$,
\[
	HF^\infty_\kk(Y, \mft) \cong \Lambda^{\ast}H^1\otimes_\Z \kk[U, U^{-1}]
\]
as $\Lambda^{\ast}H_{1}\otimes_\Z \kk[U]$-modules. The group $\Lambda^* H^1$ is graded by setting $\mathrm{gr}(\Lambda^{b_1(Y)}H^1(Y;\Z)) = b_1(Y)/2$ and by letting the action of $H_1(Y; \Z)/\textup{Tors}$ by contraction drop gradings by one. 

Let $M$ be any $\Lambda^{\ast}(H) \otimes \kk$-module. The kernel of the action of $\Lambda^{\ast}(H) \otimes \kk$ on $M$ is
\[
	\mathcal{K}M:=\{ x\in M \mid v\cdot x=0 \quad  \forall  \ v\in H \otimes \kk \}.
\]
The quotient of $M$ by this action is defined by
\[
	\mathcal{Q}M:= M/ (\mathcal{I} \cdot M).
\]
where $\mathcal{I}$ is the two-sided ideal in $\Lambda^{\ast}(H) \otimes \kk$ generated by $H$. For a standard three-manifold $Y$, there are then induced maps:
\begin{eqnarray*}
\mathcal{K}(\pi): &\mathcal{K}HF^{\infty}_\kk(Y, \mft) \rightarrow \mathcal{K}HF^{+}_\kk(Y, \mft) \\
\mathcal{Q}(\pi): &\mathcal{Q}HF^{\infty}_\kk(Y, \mft) \rightarrow \mathcal{Q}HF^{+}_\kk(Y, \mft).
\end{eqnarray*}

We may now define the \emph{bottom} and \emph{top correction terms} of $(Y, \mft)$ to be the minimal grading of any nonzero element in the image of $\mathcal{K}(\pi)$ and $\mathcal{Q}(\pi)$, denoted by $d_{bot,\kk}$ and $d_{top,\kk}$,  respectively. 

\begin{proposition}[Ozs\'vath-Szab\'o, \cite{OS03}]
\label{prop:d-inv}
Let $K\subset Y$ be a nullhomologous knot in a three-manifold $Y$ with $b_{1}(Y)\leq 1$. 
Then 
\[ 
	d_{top,\kk}(Y_{0})-\dfrac{1}{2}\leq d_{top,\kk}(Y_{1}). 
\]
Further, if $HF_{red,\kk}(Y)=0$, then $d_{top,\kk}(Y_{0})-\dfrac{1}{2}= d_{top,\kk}(Y_{1})$. 
\end{proposition}

Note that if $b_1(Y) \leq 1$, then $b_{1}(Y_{0})\leq 2$ and $b_{1}(Y_{1})\leq 1$ and hence these manifolds both have standard $HF^{\infty}$ \cite{OS03}.  Therefore, our use of $d_{bot}$ and $d_{top}$ is justified.

It is natural to ask about three-manifolds which have non-standard $HF^\infty$.  Links with $\mu_{123}\neq 0$ produce a supply of three-manifolds which are not standard.

\begin{theorem}
\label{thm: rank 6}
Let $\L$ be an algebraically split link with three components.  Suppose that $\mu_{123}(\L)\neq 0$. Then $HF^{\infty}_\mathbb{Q}(S^3_{0,0,0}(\L), \mathfrak{s}_0)$ is free of rank 6 over $\mathbb{Q}[U,U^{-1}]$, where $\mathfrak{s}_0$ is the unique torsion Spin$^c$ structure.
\end{theorem}
\begin{proof}
Recall that since $\L$ has pairwise linking number zero, there exists a basis $a_1, a_2, a_3$ for $H^1(S^3_{0,0,0}(\L);\Z)$ such that the multiplicity of the triple cup product on cohomology is given by $\mu_{123}$.  It now follows immediately from \cite[Proposition 35.3.2]{KM} that $\overline{HM}_\mathbb{Q}(S^3_{0,0,0}(\L), \mathfrak{s}_0)$ is free of rank 6 over $\mathbb{Q}[U,U^{-1}]$.  Since Heegaard and monopole Floer homology are isomorphic over $\mathbb{Z}$ by \cite{KLT1, KLT2, KLT3, KLT4, KLT5} or \cite{CGH1, CGH2, CGH3, Taubes}, we have the same result for $HF^\infty$.  
\end{proof}

\begin{remark}
\label{rem: rank 6 over k}
The same proof shows that if $\kk$ is an arbitrary field and $\mu_{123}(\L)$ is coprime to the characteristic of $\kk$ then $HF^{\infty}_{\kk}(S^3_{0,0,0}(\L), \mathfrak{s}_0)$ is free of rank 6  over $\kk[U,U^{-1}]$.
\end{remark}

\subsection{$d$-invariant inequalities}

In this section we collect some inequalities for $d$-invariants of surgeries of links over an arbitrary field $\kk$.  The key result is Proposition \ref{prop: L space sublink} which shows that the $d_{\kk}$-invariant of $S^3_{\bm{1}}(\L)$ detects the unlink when $\L$ is an L--space link.

We recall from \cite{OS03} that if $(W,\mathfrak{z})$ is a negative-definite $\spinc$-cobordism from $(Y,\mfs)$ to $(Y',\mfs')$, two rational homology spheres, then 
\begin{equation}\label{eq:neg-def-shift}
d(Y',\mfs') - d(Y,\mfs) \geq \delta(W,\mathfrak{z}).  
\end{equation}

The following three results are well-known consequences of \eqref{eq:neg-def-shift} and the formulas for the $d$-invariants of lens spaces from \cite{OS03}.  (Recall that a positive surgery on a nullhomologous knot induces a positive-definite two-handle cobordism; reversing orientation produces a negative-definite cobordism.)
 
\begin{lemma}
\label{lem: d invariant bound}
Let $\L$ be an $\ell$-component algebraically split nullhomologous link in a rational homology sphere $Y$.  Fix a $\spinc$-structure $\mfs$ on $Y$.  For any $\bm{m}=(m_1,\ldots,m_{\ell})$ with $m_i>0$ we have
\begin{equation}
\label{largeY}
\dk(Y_{\bm{m}}(L),\mfs_{\bm{i}}) \leq \dk(Y, \mfs) + \sum_{k=1}^{\ell} d(L(m_k,1),i_k).
\end{equation}
where $\bm{i}=(i_{1}, \cdots, i_{\ell})$.
\end{lemma}
 
\begin{corollary}
\label{coro:rational}
Let $\L$ be an algebraically split link in $S^3$ with $\ell$ components. Then for any integers $p_1,\ldots,p_{\ell}>0$ we have
\[
	d_{\kk}(S^3_{1/p_1,\ldots,1/p_\ell}(\L))\leq d_{\kk}(S^3_{\bm{1}}(\L)).
\]
\end{corollary}

\begin{corollary}
\label{link and sublink}
Let $\L'$ be a sublink in $\L$. Then
\[
d_{\kk}(S^3_{\bm{1}}(\L))\le d_{\kk}(S^3_{\bm{1}}(\L')).
\]
\end{corollary}

In what follows, it will be important to be able to relate the $d$-invariants of $+1$-surgery and large surgery.  The key lemma we need is the following.  
\begin{lemma}\label{lem:1<n}
Let $K$ be a nullhomologous knot in a rational homology sphere $Y$ with $\spinc$ structure $\mfs$.   Then, for $n \gg 0$, we have
\begin{equation}\label{eq:Y1Yn}
d_{\kk}(Y_1(K),\mfs_0) \leq d_{\kk}(Y_n(K),\mfs_0) - d(L(n,1), 0).
\end{equation}  
\end{lemma}  

\begin{proof}
For $n > 0$, consider the cobordism $W:Y \to Y_1(K)$ given by $W_1(K) \#_{n-1} \mathbb{C}P^2$.  Of course, this is a positive-definite four-manifold with diagonalizable intersection form $n \langle 1 \rangle$.  We will work with $-W$.  Choose a basis $\alpha_1,\ldots,\alpha_n$ for $H_2(W)$, where $\alpha_1 = [F_1]$ and the other $\alpha_i$ are given by the exceptional spheres $S_{i}$ in the $\mathbb{C}P^2$'s.  Here, $F_1$ is a capped off Seifert surface for $K$.  To pin down the signs more carefully, consider the obvious Kirby diagram for $-W$ and handleslide the $n-1$ many $-1$-framed unknots onto the $-1$-framed copy of $-K$.  We now choose the signs on $[F_1]$ and $[S_i]$ so that $[F_n] = \sum_{i=1}^n \alpha_i$ which corresponds to the Seifert surface of the knot $K$ with framing $n$.  (This is pinned down up to an overall sign by the class of $[F_n]$, which will not matter.)

Next, let $\mathfrak{z}$ denote the $\spinc$ structure on $-W$ which evaluates to one on each of the basis elements.  This is the $\spinc$ structure for which $c_1(\mathfrak{z})^2$ is maximized, i.e. $c_1(\mathfrak{z})^2 = -n$.  Note that we can break $-W$ up into two cobordisms $X_1 : -Y \to -Y_n(K)$ and $X_2: -Y_n(K) \to -Y_1(K)$.  Of course, each $X_i$ is still negative definite.  Let $\mathfrak{z}_i$ denote the restriction of $\mathfrak{z}$ to $X_i$.  

Therefore, we have from \eqref{eq:neg-def-shift}:
\begin{equation}\label{eq:d-inequality-Y1Yn}
d_{\kk}(Y_1(K), \mfs_0) \leq d_{\kk}(Y_n(K), \mathfrak{z}\mid_{Y_n(K)}) - \delta(X_2, \mfz_2).
\end{equation}

The result will then be complete if we can establish two results.  First, we want to see that $\mathfrak{z} \mid_{Y_n(K)} = \mfs_0$.  Second, we want to compute that $\delta(X_2, \mfz_2) = d(L(n,1),0)$
Since $F_n$ is supported in $X_1$ and $[F_n] = \sum_i \alpha_i$, we see that 
\[
	\langle c_1(\mathfrak{z}_1), [F_n] \rangle = \langle c_1(\mathfrak{z}), [F_n] \rangle = n.
\]     
The last equality follows since $c_1(\mathfrak{z})$ evaluates to 1 on each $\alpha_i$.  Since $H^2(X_1) = \mathbb{Z}$, $\mathfrak{z}_1$ is determined by the evaluation of the first Chern class on $[F_n]$.  Therefore, it follows that $\mathfrak{z}_1 \mid_{-Y_n(K)}  = \mfs_0$.  (Recall that Spin$^c$ structures do not require an orientation to define, so we can equate the Spin$^c$ structures on $Y_n(K)$ and $-Y_n(K)$.)  

Therefore, it remains to compute $\delta(X_2, \mfz_2)$. By our choice of $\mathfrak{z}$, 
\[
4\delta(W, \mfz) = c_1(\mathfrak{z})^2 - 3\sigma(W) - 2\chi(W) = 0.  
\]
Next, note that $c_1(\mathfrak{z}_1)^2 - 3\sigma(X_1) - 2\chi(X_1) = -c_1(\mathfrak{z}_2)^2 + 3\sigma(X_2) + 2\chi(X_2)$, since each of these terms is additive over gluing along rational homology spheres.  Therefore, we will compute $c_1(\mathfrak{z}_1)^2 - 3\sigma(X_1) - 2\chi(X_1)$ instead.  This is easy to compute, since the intersection form of $X_1$ is $\langle - n \rangle$.  Since $c_1(\mathfrak{z}_1)$ is $n$ times the generator of $H^2(X_1)$, we have $c_1(\mathfrak{z}_1)^2 = -\frac{n^2}{n} = -n$.  Therefore, we see that 
\[
\delta(X_1, \mfz_1) = \frac{c_1(\mathfrak{z}_1)^2 - 3\sigma(X_1) - 2\chi(X_1)}{4} = \frac{1-n}{4},
\]  
which is exactly $-d(L(n,1),0)$ and we are now done by \eqref{eq:d-inequality-Y1Yn}.  
\end{proof}
 
\begin{corollary}
For any algebraically split link $\L$ in $S^3$ we have
\begin{equation}\label{eq:111<nnn}
d_{\kk}(S^3_{\bm{1}}(L)) \leq d_{\kk}(S^3_{\bm{n}}(L),\mfs_{\bm{0}}) - \sum_{i=1}^\ell d(L(n_i,1), 0),
\end{equation} 
where $\bm{n} = (n_1,\ldots,n_\ell)$ is chosen to be a sufficiently large surgery, and $\mfs_{\bm{0}}$ denotes the unique trivial $\spinc$ structure. 
\end{corollary}
  
The following lemma is a straightforward analogue of \cite[Theorem 2.5]{NiWu}.
\begin{lemma}\label{lem:knot-monotonicity}
Let $K$ be a nullhomologous knot in a rational homology sphere $Y$.  Choose $m$ a large positive integer and fix a $\spinc$ structure $\mfs$ on $Y$.  Suppose that $\dk(Y_m(K),\mfs_0) = \dk(Y,\mfs) + d(L(m,1),0)$.  
Then $\dk(Y_m(K), \mfs_i) = \dk(Y,\mfs) + d(L(m,1), i)$ for each $i$.
\end{lemma}
Note that $\overline{\mfs_i} = \overline{\mfs}_{-i}$.  
\begin{proof}
First, we assume that $0 \leq i \leq \frac{m}{2}$.  By the large surgery formula from \cite{OS:Hol} and absolute gradings on the mapping cone formula \cite[Section 7.2]{OS:Rational}, $\dk(Y_m(K), \mfs_i) = \dk(Y, \mfs) + d(L(m,1),i) - 2H_{\mfs}(i)$, where $H_{\mfs}(i)$ is defined as follows.  The map $v^+_{\mfs,i} : H_*(A^+_{\mfs,i}) \to HF^+(Y,\mfs)$ defined in \cite{OS:Integer} is given by multiplication by a power of $U$ when restricted to the image of large powers of $U$ in $H_*(A^+_{\mfs,i})$.  This exponent is $H_{\mfs}(i)$.  Rasmussen shows $H_{\mfs}(i) \geq H_{\mfs}(i+1) \geq 0$ in \cite[Proposition 7.6]{Ras:Thesis} and the result follows.   

If $i\le 0$ we use the conjugation symmetry between $\mfs_{i}$ and $\overline{\mfs}_{-i}$.
\end{proof}

In fact, we can generalize Lemma~\ref{lem:knot-monotonicity} to links by an induction argument.
\begin{lemma}
\label{lem:link-monotonicity}
Let $\L$ be a nullhomologous algebraically split $\ell$-component link in a rational homology sphere $Y$.  Choose an $\ell$-tuple of large positive integers $\bm{m}=(m_1,\ldots,m_\ell)$ and fix a $\spinc$ structure $\mfs$ on $Y$.  Suppose 
\begin{equation*}
\dk(Y_{\bm{m}}(\L), \mfs_{\bm{0}}) = \dk(Y,\mfs) + \sum_{k=1}^{\ell} d(L(m_k,1),0).
\end{equation*}
Then, $\dk(Y_{\bm{m}}(\L),\mfs_{\bm{i}}) = \dk(Y,\mfs) + \sum_{k=1}^{\ell} d(L(m_k,1),i_k)$ for any tuple $\bm{i} = (i_1,\ldots,i_\ell)$.  
\end{lemma}
\begin{proof}
We prove this by induction on the number of components in a link in an arbitrary rational homology sphere.  If $\L$ is a knot, this is simply Lemma~\ref{lem:knot-monotonicity}.  

Next, suppose that we have established the result for $\ell$-component links in an arbitrary rational homology sphere, and let $\L$ be an $(\ell + 1)$-component link in $Y$.  Let $\L' = \L - L_{\ell+1}$. We will consider two decompositions of the two-handle cobordism $W_{m_1,\ldots, m_{\ell+1}}$ from $Y$ to $Y_{\bm{m}}(\L)$,
\[
	W_{m_{\ell+1}} \cup X_{m_1, \cdots, m_\ell}: Y \rightarrow Z \rightarrow Y_{\bm{m}}(\L),
\]
\[
	W_{m_1, \cdots, m_\ell} \cup X_{m_{\ell+1}}: Y \rightarrow Z' \rightarrow Y_{\bm{m}}(\L)
\]
both shown in Figure \ref{cobordism}. Here $Z = Y_{m_{\ell+1}}(L_{\ell+1})$ and $Z'=Y_{\bm{m}'}(\L')$. Again, the subscripts of the $W$- and $X$-labeled cobordisms indicate the components and framings for which the two-handles are attached, as in the notation of the beginning of the section.  Let  $\overline{\L'}$ be the image of $\L'$ in $Z$, and $\overline{L_{\ell+1}}$ be the image of $L_{\ell+1}$ in $Z'$.  The $\spinc$ structures on $Y, Z, Z'$, and $Y_{\bm{m}}(\L)$ are denoted as in Figure \ref{cobordism} where $\mfs_{\bm{i}'}=\mfs_{i_1,\cdots, i_\ell}$ and $\mfs_{\bm{i}} = \mfs_{i_1,\cdots, i_{\ell+1}}$.

\begin{figure}[h]
\centering
\begin{tikzpicture}
    \node[anchor=south west,inner sep=0] at (0,0) {\includegraphics[width=4.5in]{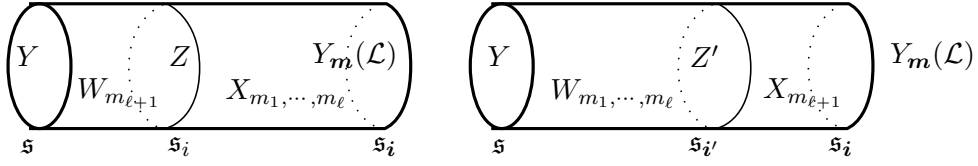}};
    \draw(0.3, 1) node {$Y$};
    \draw(0.3, -0.2) node {$\mfs$};    
    \draw(1.5, 0.5) node {$W_{m_{\ell+1}}$};
    \draw(2.3, 1) node {$Z$};
    \draw(2.3, -0.2) node {$\mfs_{i}$};     
    \draw(3.7, 0.5) node {$X_{m_1,\cdots, m_\ell}$};
    \draw(4.6, 1) node {$Y_{\bm{m}}(\L)$};
    \draw(5, -0.2) node {$\mfs_{\bm{i}}$};    
    \draw(6.5, 1) node {$Y$};
    \draw(6.5, -0.2) node {$\mfs$};     
    \draw(8, 0.5) node {$W_{m_1,\cdots, m_\ell}$};
    \draw(9.2, 1) node {$Z'$};
    \draw(9.2, -0.2) node {$\mfs_{\bm{i}'}$};       
    \draw(10.5, 0.5) node {$X_{m_{\ell+1}}$};
    \draw(12.2, 1) node {$Y_{\bm{m}}(\L)$};
    \draw(11, -0.2) node {$\mfs_{\bm{i}}$};                                 							
\end{tikzpicture}
\caption{Two different decompositions of the two-handle cobordism $W_{m_1,\ldots,m_{\ell+1}}$ from $Y$ to $Y_{\bm{m}}(\L)$ used in Lemma \ref{lem:link-monotonicity}. The left cobordism consists of first attaching the two-handle to $L_{\ell+1}$ and then attaching the remaining two-handles.  The right cobordism consists of first attaching the two-handles along all components other than $L_{\ell+1}$ and then attaching a two-handle along $L_{\ell+1}$}
\label{cobordism}
\end{figure} 

We have  
\begin{align}
\nonumber \dk(Y,\mfs) + \sum_{k=1}^{\ell+1} d(L(m_k,1),0) & = \dk(Y_{\bm{m}}(\L),\mfs_{\bm{0}}) \\
\label{eq:d-surgery-s0} & \leq \dk(Z', \mfs_{\bm{0}'}) + d(L(m_{\ell+1},1),0) \\ 
\nonumber & \leq \dk(Y,\mfs) + \sum_{k=1}^{\ell+1} d(L(m_k,1),0).
\end{align}
Similarly, 
\begin{align}
\nonumber \dk(Y,\mfs) + \sum_{k=1}^{\ell+1} d(L(m_k,1),0) & = \dk(Y_{\bm{m}}(\L),\mfs_{\bm{0}}) \\
\label{eq:d-surgery-t} & \leq \dk(Z, \mfs_0) + \sum^{\ell}_{k=1} d(L(m_k,1),0) \\ 
\nonumber & \leq \dk(Y,\mfs) + \sum_{k=1}^{\ell+1} d(L(m_k,1),0).
\end{align}

Combining \eqref{eq:d-surgery-s0} and \eqref{eq:d-surgery-t}, we conclude 
\begin{align}
\label{eq:d-Y'-pre} \dk(Z',\mfs_{\bm{0}'}) &= \dk(Y,\mfs) + \sum_{k=1}^{\ell} d(L(m_k,1),0) \\
\label{eq:d-L} \dk(Y_{\bm{m}}(\L),\mfs_{\bm{0}})  &= \dk(Z,\mfs_0) + \sum_{k=1}^{\ell} d(L(m_k,1),0).
\end{align}
By \eqref{eq:d-surgery-t} we get $\dk(Z,\mfs_0)=\dk(Y,\mfs)+d(L(m_{\ell+1},1),0)$. By applying Lemma~\ref{lem:knot-monotonicity} to the knot $L_{\ell+1}$, we also have
\begin{align}
\label{eq:d-Z} \dk(Z,\mfs_i) &= \dk(Y,\mfs) + d(L(m_{\ell+1},1),i).
\end{align}

Now, using \eqref{eq:d-Y'-pre}, by our induction assumption applied to $\L'$ in $Y$, 
\begin{equation}
\label{eq:d-inequality-surgery-induct} \dk(Z',\mfs_{\bm{i}'}) = \dk(Y,\mfs) + \sum^\ell_{k=1} d(L(m_k,1),i_k).  
\end{equation}

Similarly, by \eqref{eq:d-L}, our induction assumption applied to $\overline{\L'}$ in $Z$ implies that 
\begin{equation}
\label{eq:d-Y'} \dk(Y_{\bm{m}}(L),\mfs_{\bm{i}',0}) = \dk(Z,\mfs_0) + \sum_{k=1}^\ell d(L(m_k,1),i_k).
\end{equation}

Using \eqref{eq:d-Z} and \eqref{eq:d-Y'}, we have 
\begin{align*}
\dk(Y,\mfs) + \sum_{k=1}^{\ell} d(L(m_k,1), i_{k})+d(L(m_{\ell+1}, 1), 0) & \stackrel{\eqref{eq:d-Z}}{=} \dk(Z,\mft_0) + \sum_{k=1}^\ell d(L(m_k,1),i_k)\\
& \stackrel{\eqref{eq:d-Y'}}{=} \dk(Y_{\bm{m}}(\L),\mfs_{\bm{i'},0})\\ 
& = \dk(Z'_{m_{\ell+1}}(\overline{L_{\ell+1}}),\mfs_{\bm{i}',0}) \\
& \stackrel{\eqref{largeY}}{\leq} \dk(Z',\mfs_{\bm{i'}}) + d(L(m_{\ell+1},1),0) \\
& \stackrel{\eqref{eq:d-inequality-surgery-induct}}{\leq} \dk(Y,\mfs) + \sum_{k=1}^{\ell} d(L(m_k,1), i_{k})+d(L(m_{\ell+1}, 1), 0),
\end{align*}
hence 
\[
\dk(Z'_{m_{\ell+1}}(\overline{L_{\ell+1}}),\mfs_{\bm{i}',0}) = \dk(Z',\mfs_{\bm{i'}}) + d(L(m_{\ell+1},1),0). 
\]
By Lemma~\ref{lem:knot-monotonicity} applied to $\overline{L_{\ell+1}}$ in $Z'$, we have 
\[
\dk(Z'_{m_{\ell+1}}(\overline{L_{\ell+1}}),\mfs_{\bm{i}',i_{\ell+1}}) = \dk(Z',\mfs_{\bm{i'}}) + d(L(m_{\ell+1},1),i_{\ell+1}).
\]
Since $Z'_{m_{\ell+1}}(\overline{L_{\ell+1}}) = Y_{\bm{m}}(\L)$, \eqref{eq:d-inequality-surgery-induct} now completes the proof.
\end{proof} 

We are now ready to prove the unlink detection theorem for L--space links.

\begin{proof}[Proof of Proposition \ref{prop: L space sublink}]
We prove this by induction on the number of components.  First, suppose that $\L$ is a knot.  Then, it is well known that $\dk(S^3_{\bm{1}}(\L)) = 0$ if and only if $\L$ is unknotted.

Next, suppose that the result holds for $\ell$-component L--space links and let $\L$ have $\ell+1$ components.  Suppose that $\dk(S^{3}_{\bm{1}}(\L))=0$.  Given $\bm{m}\ggcurly \bm{0}$, by \eqref{eq:111<nnn} one has  
\[
0 = \dk(S^3_{\bm{1}}(\L)) \leq \dk(S^3_{\bm{m}}(\L), \mfs_0) - \sum_{k=1}^{\ell+1} d(L(m_k,1), 0) \leq d(S^3)=0.  
\]
The last inequality follows from Lemma \ref{lem: d invariant bound}. 

By Lemma~\ref{lem:link-monotonicity}, we have that 
\[
	\dk(S^3_{\bm{m}}(\L),\mfs_{\bm{i}}) = \sum_{k=1}^{\ell+1} d(L(m_k,1),i_k)
\] 
for all tuples $(i_1,\ldots,i_{\ell+1})$.  Note that if $\dk(S^3_{\bm{1}}(\L)) = 0$, then by Lemma \ref{link and sublink} the same is true for $\L' = L_1 \cup \ldots \cup L_\ell$.  Therefore, we have that $\L'$ is an $\ell$-component unlink by the assumption. Hence we have that $S^3_{\bm{m'}}(L')$ is  $\#_{k=1}^{\ell} L(m_k,1)$ which is an L-space.  Since $\bm{m} \ggcurly \bm{0}$, we have that $S^3_{\bm{m}}(\L)$ is an L--space.  Therefore, the image of $L_{\ell+1}$ in $S^3_{\bm{m'}}(\L')$ is a knot for which $m_{\ell+1}$-surgery yields a three-manifold with the same absolutely graded Heegaard Floer homology as $S^3_{\bm{m'}}(\L') \# L(m_{\ell+1}, 1)$.  By Gainullin's Dehn surgery characterization of the unknot in L-spaces \cite[Theorem 8.2]{Gainullin}, we have that $L_{\ell+1}$ is unknotted in $S^3_{\bm{m'}}(\L')$.  (His paper is only written over $\F$, but the arguments work over an arbitrary field.)  By varying the values of $\bm{m'}$, we may apply \cite[Corollary 2.4]{GabaiII} to conclude that $L_{\ell+1}$ is actually unknotted in the exterior of $\L'$.  Since $\L'$ is also an unlink, we see that $\L$ is an unlink.   
 The second part directly follows from Corollary \ref{coro:rational}.
\end{proof}

\begin{remark}
\label{rem:grading}
There is a very elementary proof in the case that $\kk = \F$.  One simply uses \eqref{eq:111<nnn} to see
\begin{equation}
\label{shift}
d_{\F}(S^3_{\bm{1}}(\L)) \leq d_{\F}(S^3_{\bm{m}}(L), \mfs_{\bm{0}}) - \sum^\ell_{k = 1} d(L(m_k,1),0) = -2h_{\L}(\bm{0}) \leq -2.
\end{equation}
Recall that $-2H_{\L}(\bm{0})$ equals $d_{\F}(S^3_{\bm{m}}(L), \mfs_{\bm{0}})$ up to grading shift which does not depend on the link. Hence one can use the unlink to cancel the grading shift, obtaining the equality in \eqref{shift}. The last inequality follows from the fact that a nontrivial link $\L$ has $h_{\L}(\bm{0})> 0$ \cite{Beibei}.   
\end{remark}

\subsection{A bound from non-vanishing triple linking}

In order to constrain the $d$-invariants of $S^3_{\bm{1}}(\L)$ in the case that $\L$ is a three-component link with non-trivial Milnor triple linking, we will connect this with the Floer homology of $S^3_{\bm{0}}(\L)$ which we know is {\em not} standard by Theorem~\ref{thm: rank 6}.  To do this, we will use the $d$-invariant inequalities (and equalities) that come from the surgery triangle for  surgery on nullhomologous knots in three-manifolds which do have standard $HF^{\infty}$.  For $n = 0, \infty$ or odd, let $\mfs_n$ denote the unique self-conjugate torsion $\spinc$ structure on $S^3_{0,0,n}(\L)$.    

\begin{lemma}
\label{dtop dbot}
Assume that $\L$ is a Brunnian link with $\mu_{123}\neq 0$. Then
\[d_{top,\Q}(S^3_{0,0,1}(\L),\mfs_1) \leq d_{bot,\Q}(S^3_{0,0,\infty}(\L),\mfs_\infty).\]
\end{lemma}

\begin{proof}
The first step of the proof is to relate the $d$-invariants of $S^3_{0,0,\infty}(\L)$ with $S^3_{0,0,n}(\L)$.  
Choose odd $n \gg 0$.  We begin by considering the surgery exact triangle from \cite[Theorem 9.19]{OS:properties}: 
\begin{equation*}
\label{eq:triangle}
\xymatrix @M=6pt@C=10pt@R=10pt {
	HF^\infty_{\Q}(S^3_{0,0,\infty}(\L), \mfs_\infty) \ar[rr]^{F_1} & &
	HF^\infty_{\Q}(S^3_{0,0,0}(\L),\mfs_0) \ar[dl]^{F_2} \\
	& HF^\infty_{\Q}(S^3_{0,0,n}(\L), \mfs_n) \ar[ul]^{F_3}
}
\end{equation*}
In the above exact triangle, $F_3$ is a sum of {\em two} $\spinc$ cobordism maps, and by our choice of $\mfs_n$, these have the same absolute grading shift, given by $-d(L(n,1),0)$ (See for instance \cite[Section 4.8]{OS:Integer}).  Furthermore, $F_3$ preserves the absolute $\Z/2$-grading defined in \cite{OS:properties}.        

Since $\L$ is Brunnian, $S^3_{0,0,\infty}(\L)$ is $\#^2(S^1\times S^2)$. Furthermore, for $Y=S^3_{0,0,n}(\L)$ or $S^3_{0,0,\infty}(\L)$, since $b_1(Y) \leq 2$, the module
\begin{equation}\label{eq:standard}
	HF^\infty_{\Q}(Y,\mfs) \cong \Lambda^{\ast}H^1(Y;\Z) \otimes_\Z \Q[U, U^{-1}]
\end{equation}
is standard for any torsion $\spinc$ structure $\mfs$.  The elements of $H_1(Y;\Z)/\text{Tors}$ act by contraction on $\Lambda^\ast H^1(Y;\Z)$, hence on $HF^\infty_\Q(Y,\mfs)$. 

Choose generators $x,y$ of $H_1(S^3_{0,0,n}(\L);\Z)/\text{Tors}$.  Let us choose a $\Q[U,U^{-1}]$-basis for $HF^\infty_{\Q}(S^3_{0,0, n}(\L), \mfs_n)$, denoted $\alpha, \beta, \gamma, \delta$, which correspond to $x^* \wedge y^*, x^*, y^*, 1$ respectively under the isomorphism in \eqref{eq:standard}.  Likewise, choose generators $x', y'$ for $H_1(S^3_{0,0,\infty}(\L);\Z)/\text{Tors}$ which are bordant to $x,y$ in the surgery cobordism from $S^3_{0,0,n}(\L)$ to $S^3_{0,0,\infty}(\L)$; define the analogous generators of $HF^\infty_{\Q}(S^3_{0,0,\infty}(\L), \mfs_\infty)$ by $\alpha', \beta', \gamma', \delta'$. Then $F_3$ may be expressed as follows
\begin{equation*}
\label{eq:triangle}
\xymatrix @M=6pt@C=2pt@R=5pt {
	\alpha = \langle x^*\wedge y^* \otimes 1 \rangle   &&& \alpha' = \langle x'^*\wedge y'^* \otimes 1 \rangle\\
	\beta = \langle x^*\otimes 1 \rangle  \oplus  \langle y^*\otimes 1\rangle = \gamma  \ar@/^2pc/[rrr]^{F_3} &&& 
	\beta' = \langle x'^*\otimes 1 \rangle \oplus  \langle y'^*\otimes 1\rangle =\gamma'\\
	\delta = \langle 1 \otimes 1\rangle &&& \delta' = \langle 1 \otimes 1\rangle 
}
\end{equation*} 

By Theorem \ref{thm: rank 6} $HF^\infty_{\Q}(S^3_{0,0,0}(\L), \mfs_0)$ is of rank 6. Exactness implies that 
\[
	HF^\infty_{\Q}(S^3_{0,0,0}(\L), \mfs_0) \cong \operatorname{coker} F_3 \oplus \ker F_3, 
\]
and so each of $\ker F_3$ and $\operatorname{coker} F_3$ is rank 3. Recall that the map $F_{3}$ is equivariant with respect to the action of the exterior algebra, being a sum of cobordism maps. Hence, the image of $\alpha$ determines the map $F_{3}$. We claim that $F_{3}(\alpha)$ has a component which is a non-zero multiple of $\delta'$.  Before proving the claim, let us see why this will complete the proof.  It follows from the long exact sequence relating $HF^-_\Q, HF^\infty_\Q$ and $HF^+_\Q$ that 
\[
d_{top,\Q}(S^3_{0,0,n}(\L),\mfs_n) - d(L(n,1),0) \leq d_{bot,\Q}(S^3_{0,0,\infty}(\L),\mfs_\infty).
\]
The arguments of Lemma~\ref{lem: d invariant bound} establish 
\[
d_{top,\Q}(S^3_{0,0,1}(\L),\mfs_1) \leq d_{top,\Q}(S^3_{0,0,n}(\L),\mfs_n) - d(L(n,1),0),
\]
which completes the proof.
 
Therefore, it remains to prove that $F_3(\alpha)$ contains a non-zero multiple of $\delta'$.  Suppose instead that $F_3(\alpha)$ is a non-trivial linear combination of $\alpha'$, $\beta'$ and $\gamma'$.  

First, since $F_3$ is equivariant with respect to the action of the exterior algebra, if $F_3(\alpha) = 0$, then $F_3$ is identically 0. Thus $\ker F_3$ has rank 4, a contradiction.   
Since $F_3$ respects the $\Z/2$-grading, $\beta'$ and $\gamma'$ cannot be components of $F_3$.  Therefore, it remains to assume $F_3(\alpha)=c\alpha'+d\delta'$ for a unit $c$ (which might involve a nonzero rational factor and a power of $U$). Again, since $F_3$ is equivariant with respect to the action the exterior algebra, contraction by $y$ implies that
\begin{eqnarray*}
F_3(\beta) &=& F_3(x^*\otimes1) = -F_3\circ \iota_y (x^*\wedge y^* \otimes 1)	\\
 	&=&-\iota_{y'} \circ F_3 (x^*\wedge y^* \otimes 1) = -c\cdot \iota_{y'}( x'^*\wedge y'^* \otimes 1)  = cx'^*\otimes 1 = c \beta'.
\end{eqnarray*}
Similarly, $F_3(\gamma) = c \gamma'$ and $F_3(\delta)=c\delta'$. Thus, $\rank\ker(F_3)=0$, which is again a contradiction. 
\end{proof}

\begin{remark}\label{rem: coprime}
By Remark \ref{rem: rank 6 over k}, the above proof works over an arbitrary field $\kk$ if $\mu_{123}(\L)$ is coprime to the characteristic of $\kk$.
\end{remark}

Now we can prove:
\begin{reptheorem}{nonzero linking d invt}
Let $\L=L_1\cup L_2\cup L_3$ be an algebraically split link such that all two-component sublinks are $\Q$-L--space links. If the triple linking number $\mu_{123}$ is nonzero, then $d_{\Q}(S^3_{1,1,1} (\L))\leq -2$. If the triple linking number $\mu_{123}$ is odd, then the analogous statement holds with $\mathbb{Z}_2$-coefficients.
\end{reptheorem}
\begin{proof}
Let $\L$ be an algebraically split three-component link such that all two-component sublinks are L--space links. 
If one of these sublinks is nontrivial, then the result follows from Corollary \ref{link and sublink} and Proposition \ref{prop: L space sublink}. Therefore from now on we can assume that all two-component sublinks are trivial, so $\L$ is a Brunnian link.

Since $S^3_{0,0,\infty}(\L)=S^3_{0,0}(L_{1}\cup L_{2})=\#^2(S^1\times S^2)$, we have $d_{bot,\Q}(S^3_{0,0,\infty}(\L))=-1$.
By Lemma \ref{dtop dbot} we get $$d_{top,\Q}(S^3_{0,0,1}(\L))\le d_{bot,\Q}(S^3_{0,0,\infty}(\L))=-1.$$ On the other 
hand, $S^3_{0,\infty,1}(\L) = S^2 \times S^1$ and $S^3_{\infty, 1, 1}(\L) = S^3$, so by Proposition \ref{prop:d-inv} we get
$$
d_{top,\Q}(S^3_{0,0,1}(\L))=d_{top,\Q}(S^3_{0,1,1}(\L))+\frac{1}{2}=d_{\Q}(S^3_{1,1,1}(\L))+\frac{1}{2}+\frac{1}{2},
$$
and we conclude $d_{\Q}(S^3_{1,1,1}(\L))\le -2$.  A similar argument applies for the case of $\mathbb{Z}_2$ coefficients by Remark~\ref{rem: coprime}.  (Alternatively, see Corollary~\ref{cor:d-mu123-odd} below.)
\end{proof}

\subsection{$0$-surgery on links}

In this subsection we describe a different approach to the computation of $d$-invariants of $S^3_{000}(\L)$ building on the work of the second author in \cite{Tye}.\footnote{As the article appears on arXiv and in thesis form, there is a gap in the argument for $b_1 \geq 5$.  This does not affect the arguments used here.} Since it uses the link surgery formula of \cite{MO}, we have to restrict ourselves to the coefficients in $\F=\Z_2$. 

Recall that the complex $CF^-(S^3_{000}(\L))$ in the unique torsion $\spinc$-structure can be written as in Figure \ref{fig: surgery} using the surgery formula of \cite{MO}.

\begin{figure}[ht!]
\begin{tikzpicture}
\draw(3,4) node {$\A^-_{000}(\L,\bm{0})$};
\draw(0,2) node {$\A^-_{00}(L_{12},\bm{0})$};
\draw(3,2) node {$\A^-_{00}(L_{13},\bm{0})$};
\draw(6,2) node {$\A^-_{00}(L_{23},\bm{0})$};
\draw(0,0) node {$\A^-_{0}(L_1,\bm{0})$};
\draw(3,0) node {$\A^-_{0}(L_2,\bm{0})$};
\draw(6,0) node {$\A^-_{0}(L_3,\bm{0})$};
\draw(3,-2) node {$CF^-_{\F}(S^3)$};
\draw [->] (2.8,3.8)--(0.2,2.2);
\draw [dashdotted,->] (2.8,3.8)--(0.2,0.2);
\draw [->] (3.2,3.8)--(5.8,2.2);
\draw [dashdotted,->] (3.2,3.8)--(5.8,0.2);
\draw [dashdotted,->] (.15,1.8)--(2.8, -1.8);
\draw [dashdotted,->] ( 5.9 ,1.8 )--(3.1,-1.7);
\draw [->] (2.8,1.8)--(0.3,0.2);
\draw [->] (3.2,1.8)--(5.7,0.2);
\draw [->] (3,3.8)--(3,2.2);
\draw [->] (3,-0.2)--(3,-1.8);
\draw [->] (0,1.8)--(0,0.2);
\draw [->] (6,1.8)--(6,0.2);
\draw [->] (0.2,1.8)--(2.8,0.2);
\draw [->] (5.8,1.8)--(3.2,0.2);
\draw [->] (0.2,-0.2)--(2.8,-1.8);
\draw [->] (5.8,-0.2)--(3.2,-1.8);
\draw [dashdotted,->] (3,3.8)..controls (4.2,2) and (4.2,2)..(3,0.2);
\draw [dashdotted,->] (3,1.8)..controls (1.8,0) and (1.8,0)..(3,-1.8);
\draw [dashed,->] (3.4,3.8)..controls (12,1) and (12,1)..(3.4,-1.8);
\end{tikzpicture}
\caption{Surgery complex quasi-isomorphic to $CF^-(S^3_{000}(\L))$.}
\label{fig: surgery}
\end{figure}
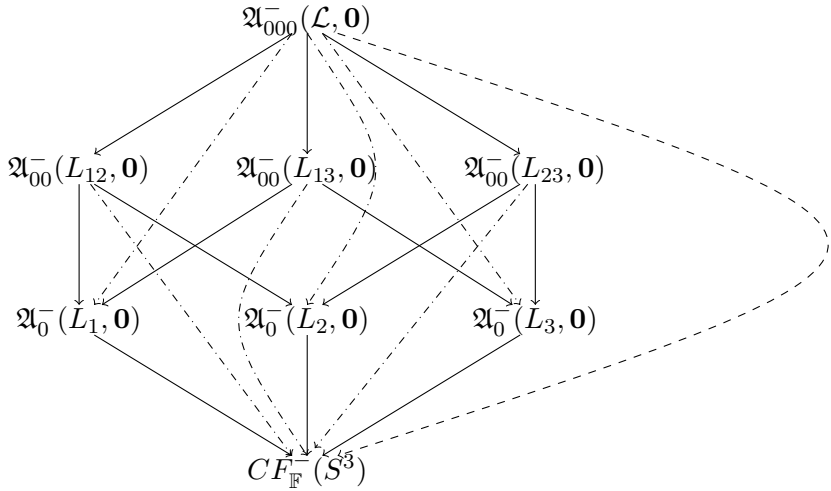

\begin{theorem}[\cite{MO}]
\label{th: sp seq}
The complex $CF^-(S^3_{000}(\L))$ in the unique torsion $\spinc$-structure is quasi-isomorphic (as a complex of free $\F[U]$-modules) to the complex
in Figure \ref{fig: surgery}. The cube filtration on this complex induces a spectral sequence where all pages are link invariants.
\end{theorem}

After tensoring the surgery complex with $\F[U,U^{-1}]$ over $\F[U]$ and using the fact that $\A^-(\L,\bm{0})\otimes_{\F[U]}\F[U,U^{-1}]$ is homotopy equivalent to $\F[U,U^{-1}]$ with trivial differential for any link $\L$, we can simplify the surgery complex for $CF^{\infty}(S^3_{000}(\L))$. In fact, we an simplify it even further.

\begin{theorem}[\cite{Tye}]
The surgery complex for $CF^{\infty}(S^3_{000}(\L),\mfs_0)$ is quasi-isomorphic to the complex in Figure \ref{fig: surgery2}.
The differentials $d_1$ and $d_2$ in the associated spectral sequence vanish, while the relevant $d_3$ differential is given (up to a unit) by multiplication by the triple linking number $\mu_{123}(\L)$ modulo 2.
\end{theorem}
 
\begin{figure}[ht!]
\centering
\begin{tikzpicture}
\draw(3,4) node {$\F[U,U^{-1}]$};
\draw(0,2) node {$\F[U,U^{-1}]$};
\draw(3,2) node {$\F[U,U^{-1}]$};
\draw(6,2) node {$\F[U,U^{-1}]$};
\draw(0,0) node {$\F[U,U^{-1}]$};
\draw(3,0) node {$\F[U,U^{-1}]$};
\draw(6,0) node {$\F[U,U^{-1}]$};
\draw(3,-2) node {$\F[U,U^{-1}]$};
\draw [dashed,->] (3.4,3.8)..controls (12,1) and (12,1)..(3.4,-1.8);
\draw (10.7,1) node {$\mu_{123}(\L)$};
\end{tikzpicture}
\caption{Surgery complex for computing $HF^{\infty}(S^3_{000}(\L),\mfs)$.}
\label{fig: surgery2}
\end{figure}
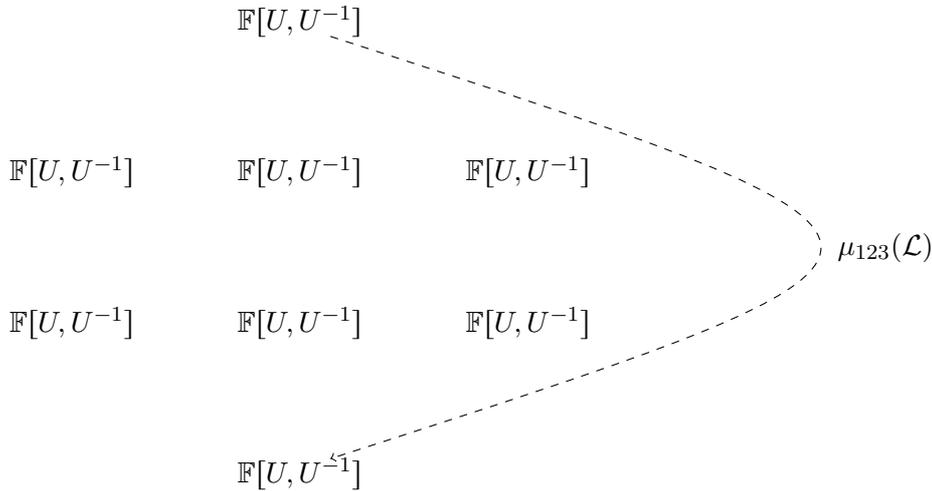

\begin{corollary}
$HF^{\infty}(S^3_{000}(\L),\mfs_0)$ has rank $8$ if $\mu_{123}(\L)$ is even and rank 6 if $\mu_{123}(L)$ is odd.
\end{corollary}
 
Recall that the homology of each $\A^-(\L, \bm{s})$ is non-canonically isomorphic to $\F[U] \oplus M$, where $M$ is annihilated by some power of $U$.  Since $U$-torsion dies after tensoring with $\F[U,U^{-1}]$, we obtain the following result. 

\begin{corollary}
\label{cor: d1 d2 on free parts}
Consider the spectral sequence from Theorem \ref{th: sp seq}. The components of the differentials $d_1$ and $d_2$ that go between free parts vanish.  Up to a power of $U$, the relevant component of the $d_3$ differential between free parts is multiplication by $\mu_{123}(\L) \mod 2$. 
\end{corollary}
 
\begin{remark}
Although the differentials $d_1$ and $d_2$ between the free parts vanish, the differentials from the free parts to the torsion are still possible, see Example \ref{ex:bing 8}.  Because $H_*(\A^-(\emptyset)) \cong \F[U]$, there is no torsion piece to which the $d_3$ differential may map. 
\end{remark} 

We can use these results to give an alternative proof of Theorem \ref{nonzero linking d invt} if the triple linking number is odd and $\kk=\mathbb{F}$.
 
\begin{theorem}
\label{thm:h}
Suppose that $\mu_{123}(L)$ is odd and all two-component sublinks of $\L$ are L--space links. Then $H_{\L}(0,0,0)>0$.
\end{theorem} 
 
\begin{proof}
Let us compute the spectral sequence for $CF^-(S^3_{000}(\L))$. Since all two-component sublinks $L_{ij}$ are L-space links,
the components $L_i$ are L--space knots, and 
$$
H_*(\A^-(L_{ij},\bm{0})) \cong  H_*(\A^-(L_i),\bm{0}) \cong \F[U].
$$
Note that the homology of $\A^-(\L,\bm{0})$ might have torsion, since we do not assume $\L$ is an L-space link. By Corollary \ref{cor: d1 d2 on free parts} the differentials $d_1$ and $d_2$ vanish on the free part of $H_*(\A^-(\A,\bm{0}))$ and have trivial image in $H_*(\A^-(\emptyset))$.  The differential $d_3$ from the free part of $H_*(\A^-(\L,\bm{0}))$ to $H_*(\A^-(\emptyset))\cong \F[U]$ is nontrivial. 

On the other hand, $d_3$ lowers the homological degree by $1$.  Further, up to an absolute shift, the generator of the free part of $H_*(\A^-(\L,\bm{0}))$ has homological degree $-2H(0,0,0)$ while the generator of the free part of $H_*(\A^-(\emptyset))$ has homological degree $-3$, so $-2H(0,0,0)-1\le -3$ and $H(0,0,0)\ge 1$.  
\end{proof}

\begin{corollary}\label{cor:d-mu123-odd}
Suppose that $\mu_{123}(L)$ is odd and all two-component sublinks of $\L$ are L--space links. Then $d(S^3_{1,1,1}(\L))\le -2$.
\end{corollary}

\begin{proof}
Similar to Remark \ref{rem:grading}, one simply uses \eqref{eq:111<nnn} to see that for $\bm{m} \gg 0$
\begin{equation}
\label{compareh}
d(S^3_{\bm{1}}(\L)) \leq d(S^3_{\bm{m}}(L), \mfs_{\bm{0}}) - \sum^3_{k = 1} d(L(m_k,1),0) = -2h_{\L}(\bm{0}) =-2H_{\L}(\bm{0})\leq -2.
\end{equation}
 The last inequality follows from Theorem \ref{thm:h}. 
\end{proof}

\subsection{Example: generalized Borromean link}

The assumption that all 2-component sublinks are L-space links is important in Theorem \ref{nonzero linking d invt}. We will show that there exist three-component algebraically split links  $\L$ with nonzero triple linking number and $d(S^{3}_{1, 1, 1}(\L))=0$. Here, we resume working exclusively over $\F$ and omit the coefficients from the notation.  

\begin{example}
\label{ex: bing}
Start with the two-component link $L_{1}=K\cup U$  in the left image of Figure \ref{Bing} where $U$ is the unknot and $K$ is arbitrary. We can assume the linking number of $L_{1}$ is $-1$. Let $L(n)=B(K, n)\cup U$ denote the new link obtained by applying an $n$-twisted Bing-double to $K$, which is the right image in Figure \ref{Bing}.  We order $B(K,n)$ so that the first component is the one ``induced'' by $K$.  Note that $L(n)$ is a three-component  algebraically split link and $\mu_{123}(L(n))=(-1) \lk(L_{1})=1$ \cite[Theorem 8.1]{Tim}.

\begin{figure}
\centering
\begin{tikzpicture}
    \node[anchor=south west,inner sep=0] at (0,0) {\includegraphics[width=3.0in]{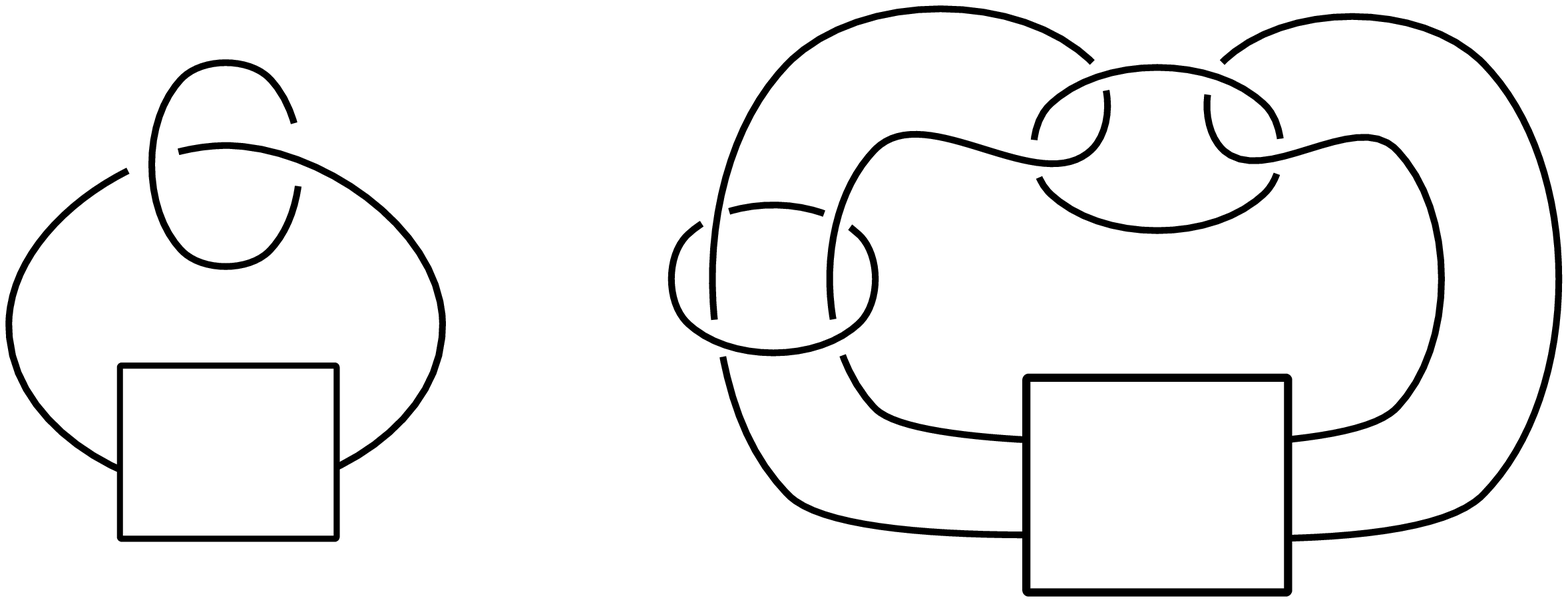}};
    \draw(1.1, 0.7) node {$K$};
    \draw(5.7, 0.5) node {$K, n$};  							
\end{tikzpicture}
\caption{The link on the left is $L_1$ from Example~\ref{ex: bing}.  The link on the right is the result of applying $n$-twisted Bing doubling to $K$ in $L_1$, yielding $L(n) = B(K,n) \cup U$.}
\label{Bing}
\end{figure}

We claim that $d(S^{3}_{1, 1, 1}(L(n)))=0$ for sufficiently large $n$. Note that $S^{3}_{1, 1, 1}(L(n))$ is diffeomorphic to $(S^{3}_{1}(D_{+}(K, n-1))$ where $D_{+}(K, n-1)$ is the $(n-1)$-twisted positively clasped Whitehead double.  The $d$-invariant for this manifold is computed in \cite{Tange} to be 
\begin{equation*}
d(S^{3}_{1}(D_{+}(K, n))) = \left\{
        \begin{array}{ll}
            0  & \quad n\geq 2\tau(K) \\
            -2 & \quad n<2\tau(K).
        \end{array}
    \right. 
\end{equation*}
Hence, for sufficiently large $n$, we have $\mu_{123}(L(n))= 1$, but $d(S^{3}_{1, 1, 1}(L(n)))=0$. 
\end{example}

\begin{remark}
In the above example, 
\[
 -\lambda(S^3_{1, 1, 1}(\L)) = -\mu^2_{123}(\L) + \beta(B(K, n)) = n-1. 
\]
We also have the alternate computation
\[
 \beta(B(K, n)) = - \lambda(S^3_{1, 1}(B(K, n)) = -\lambda(S^3_{1}(D_{+}(K, n)) = n.
\]
\end{remark}

\begin{example}
\label{ex:bing 8}
For a specific example, let $K$ be the unknot.  In our conventions $D_+(K,0)$ is the unknot, $D_+(K,-1)=T(2,3)$ and 
$D_+(K,1)$ is the figure eight knot. In particular, for $n=0$ we get $L(0)$ is the Borromean rings and $$d(S^3_{1,1,1}(L(0)))=d(S^3_{1}(D_+(K,-1)))=-2.$$ The above computation shows that for $n\ge 1$ we have $d(S^3_{1,1,1}(L(n)))=0$. 
By a sequence of inequalities similar to \eqref{compareh}, we get $H_{L(n)}(0,0,0)=0$ for $n\geq 1$.   

For $n=1$ we can also compute all differentials in the spectral sequence of Theorem \ref{th: sp seq}. 
Indeed, all components of $\L=L(1)$ are unknots and two of three two-component sublinks are unlinks. The only interesting two-component sublink is $B(K,1)$ and in order to apply Theorem \ref{th: sp seq} we need to describe $\A^-(B(K,1),\bm{0})$. Observe that the trivial component of $B(K,1)$ has genus 1 in the complement of the other component. This means that for $p\ggcurly 1$  
the $(p,1)$ surgery is large  for $B(K,1)$ and 
$$
\A^-(B(K,1),\bm{0})\simeq CF^-(S^3_{p,1}(B(K,1),\bm{s}_0) \simeq CF^-(S^3_p(D_+(K,1),\bm{s}_0)) \simeq \A^-(D_+(K,1),0).
$$ 
Here the first and last equations follow from the large surgery formula, and the middle equation is clear.

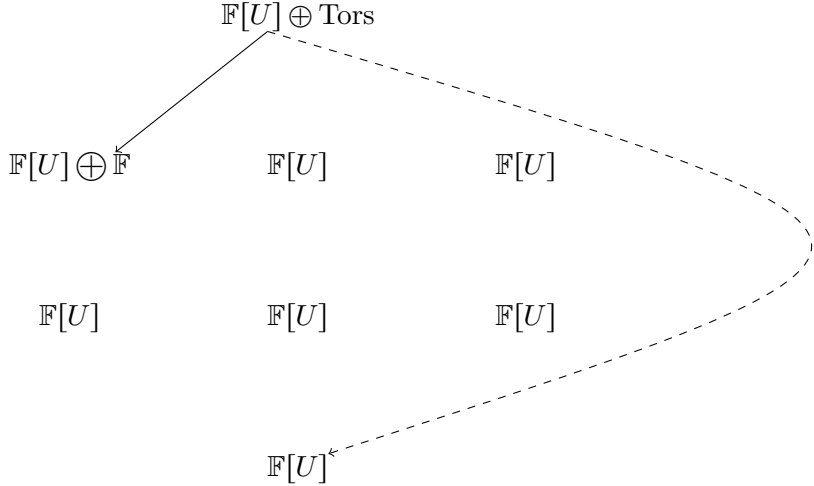
\begin{figure}[t]
\begin{tikzpicture}
\draw(3,4) node {$\F[U]\oplus \textrm{Tors}$};
\draw(0,2) node {$\F[U]\bigoplus \F$};
\draw(3,2) node {$\F[U]$};
\draw(6,2) node {$\F[U]$};
\draw(0,0) node {$\F[U]$};
\draw(3,0) node {$\F[U]$};
\draw(6,0) node {$\F[U]$};
\draw(3,-2) node {$\F[U]$};
\draw [->] (2.6,3.8)--(0.6,2.2);
\draw [dashed,->] (2.6,3.8)..controls (12,1) and (12,1)..(3.4,-1.8);
\end{tikzpicture}
\caption{Spectral sequence in Example \ref{ex:bing 8} for computing $HF^-(S^3_{0,0,0}(L(1)))$.  The solid arrow indicates $d_1$ while the dashed arrow indicates $d_3$.  These are the only non-trivial differentials in the spectral sequence after the $E_1$ page.}
\label{fig: surgery3}
\end{figure}

Since $D_+(K,1)$ is the figure eight knot, it is well known that $H_*(\A^-(D_+(K,1),0)) \cong \F[U]_{(0)} \oplus \F$.  
Therefore the $E_1$ page of the spectral sequence in Theorem \ref{th: sp seq} has the form shown in Figure \ref{fig: surgery3}.  Let $z$ be the generator of the free part of $H_*(\A^-(\L,\bm{0}))$. Then, since $H_{\L}(0,0,0) = 0$, we see $z$ has degree 3 higher than the generator of $H_*(\A^-(\emptyset))$.  We claim that $d_1(z)$ is the unique non-trivial element in the kernel of $U$.  (This description is independent of the choice of splitting.)  Indeed, suppose that instead $d_1(z)=0$.  By Corollary \ref{cor: d1 d2 on free parts}, $d_1$ is identically 0, and so is $d_2$, and the $d_3$ differential should map $z$ nontrivially to the homology of $\A^-(\emptyset)$ which is not possible by degree reasons.

Therefore, $d_1(z)$ is determined and the $d_1$ differential vanishes elsewhere by Corollary \ref{cor: d1 d2 on free parts}.  On the $E_2$ page we get a free $\F[U]$ module generated by $Uz$ together with the torsion at the top, and $\F[U]$ everywhere else, so by Corollary \ref{cor: d1 d2 on free parts} the differential $d_2$ vanishes. Now the differential $d_3$ 
sends $Uz$ to a power of $U$ times the generator of the homology of $\A^-(\emptyset)$.  For degree reasons, $d_3(Uz)$ is in fact the generator of $H_*(\A^-(\emptyset))$, and hence this pair of free modules is cancelled by the $d_3$ differential.  The $d_3$ differential vanishes elsewhere, and all other differentials vanish identically.  From the algebra, we cannot seem to determine from the spectral sequence what the torsion coming from $H_*(\A^-(\L,\bm{0}))$ is and whether it contributes to $HF_{red}(S^3_{0,0,0}(\L))$ or the free part.  We will use some topological input to complete the spectral sequence computation.

It is an easy Kirby calculus exercise to see that $S^3_{0,0,0}(\L) = \mathbb{T}^3$, and hence $HF_{red}(S^3_{0,0,0}(\L)) = 0$.  Thus, the torsion term contributes to the free part.  Further, ignoring this torsion part, the $E_\infty$ page of the spectral sequence has six towers.  Three towers come from the second-to-top filtration level and are all supported in the same gradings.  This relies on the fact that $d(S^3_{p,1}(B(K,1)) = 0$ and that the other two two-component sublinks are trivial.  It is also not hard to deduce that this topmost absolute grading is in fact $1/2$.  The remaining three towers come from the second-to-bottom filtration level and are all supported in the same gradings; their topmost relative grading is one lower than that of the other towers, and hence have topmost grading $-1/2$.  Note that this agrees with the relative-gradings on $HF^-(\mathbb{T}^3)$, and hence the torsion term cannot contribute to the free part of the Floer homology.  Consequently, the torsion term of $H_*(\A^-(\L,\bm{0}))$ is trivial, and we have completed the computation of the spectral sequence.  
\end{example}

\bibliographystyle{alpha}
\bibliography{biblio}

\end{document}